\documentclass[review]{elsarticle}
\usepackage[toc,page,title,titletoc,header]{appendix}
\usepackage{hyperref}
\usepackage{stmaryrd}
\SetSymbolFont{stmry}{bold}{U}{stmry}{m}{n}
\usepackage{amsmath}
\usepackage{amssymb}
\usepackage{amsthm}
\usepackage{tabularx}
\usepackage{subfigure}
\usepackage{epsfig}
\usepackage{indentfirst}
\usepackage{cases}
\biboptions{sort&compress}
\usepackage{graphicx}
\usepackage{epstopdf}
\usepackage{color}
\usepackage{float}
\usepackage{algorithmic}
\usepackage{bm}
\usepackage{multirow}
\allowdisplaybreaks[3]

\newcommand{\be}{\begin{equation}}
\newcommand{\ee}{\end{equation}}
\newcommand{\ba}{\begin{array}}
\newcommand{\ea}{\end{array}}
\newcommand{\bea}{\begin{eqnarray}}
\newcommand{\eea}{\end{eqnarray}}
\newcommand{\beas}{\begin{eqnarray*}}
\newcommand{\eeas}{\end{eqnarray*}}

\newtheorem{thm}{Theorem}[section]

\newtheorem{remark}{Remark}[section]

\newtheorem{algorithm}{Algorithm}[section]

\numberwithin{equation}{section}

\def\R{{\mathbb R}}

\newcommand{\p}{\partial}

\newcommand{\cE}{\mathcal E}

\renewcommand{\l}{\left}
\renewcommand{\r}{\right}

\renewcommand{\d}{\mathrm{d}}

\newcommand{\bn}{\mathbf{n}}

\newcommand\cV{{\mathcal V}}
\newcommand{\bX}{\mathbf{X}}
\newcommand{\bY}{\mathbf{Y}}

\begin{document}

\begin{frontmatter}

\title{A second-order in time, BGN-based parametric finite element method for geometric flows of curves}

\author[1]{Wei Jiang}
\address[1]{School of Mathematics and Statistics, Wuhan University, Wuhan 430072, China}
\ead{jiangwei1007@whu.edu.cn}

\author[2]{Chunmei Su\corref{3}}
\ead{sucm@tsinghua.edu.cn}
\author[2]{Ganghui Zhang}
\ead{gh-zhang19@mails.tsinghua.edu.cn}
\address[2]{Yau Mathematical Sciences Center, Tsinghua University, Beijing, 100084, China}
\cortext[3]{Corresponding author.}


\begin{abstract}
Over the last two decades, the field of geometric curve evolutions has attracted significant attention from scientific computing. One of the most popular numerical methods for solving geometric flows is the so-called BGN scheme, which was proposed by Barrett, Garcke, and N\"urnberg (J. Comput. Phys., 222 (2007), pp.~441--467), due to its favorable properties (e.g., its computational efficiency and the good mesh property). However, the BGN scheme is limited to first-order accuracy in time, and how to develop a higher-order numerical scheme is challenging. In this paper, we propose a fully discrete, temporal second-order parametric finite
element method, which integrates with two different mesh regularization techniques,  for solving geometric flows of curves. The scheme is constructed based on the BGN formulation and a semi-implicit Crank-Nicolson leap-frog time stepping discretization as well as a linear finite element approximation in space. More importantly, we point out that the shape metrics, such as manifold distance and Hausdorff distance, instead of function norms, should be employed to measure numerical errors. Extensive numerical experiments demonstrate that the proposed BGN-based scheme is second-order accurate in time in terms of shape metrics. Moreover, by employing the classical BGN scheme as mesh regularization techniques, our proposed second-order schemes exhibit good properties with respect to the mesh distribution.  In addition, an unconditional interlaced energy stability property is obtained for one of the mesh regularization techniques.

\end{abstract}




\end{frontmatter}

\section{Introduction}
Geometric flows, which describe the evolution of curves or surfaces over time based on the principle that the shape changes according to its underlying geometric properties, such as the curvature, have been extensively studied in the fields of computational geometry and geometric analysis. In particular, second-order (e.g., mean curvature flow, which is also called as curve-shortening flow for curve evolution) and fourth-order (e.g., surface diffusion flow) geometric flows  have attracted considerable interest due to their wide-ranging applications in materials science~\cite{Zhao-Jiang-Wang-Bao, Zhao-Jiang-Bao}, image processing~\cite{Aubert06}, multiphase fluids~\cite{Garcke23} and cell biology~\cite{BGN08C}. For more in-depth information, readers can refer to the recent review articles~\cite{BGN20,DDE2005}, and references provided therein.

In this paper,  we focus on three different types of geometric flows of curves: curve-shortening flow (CSF), area-preserving curve-shortening flow (AP-CSF) and surface diffusion flow (SDF). First, assume that $\Gamma(t)$ is a family of  simple closed curves in the two-dimensional plane. We consider that the curve is governed by the three geometric flows, i.e., its velocity is respectively given by
\begin{equation}\label{Geometric equation}
  \cV =\begin{cases}
  	-\kappa\mathbf{n},  &\quad \text{CSF}, \\
  	(-\kappa+\l<\kappa \r>)\mathbf{n},&\quad \text{AP-CSF}, \\
  	(\p_{ss}\kappa) \mathbf{n}, &\quad \text{SDF},
  \end{cases}
\end{equation}
where $\kappa$ is the curvature of the curve, $s$ is the arc-length,  $\l<\kappa \r>:=\int_{\Gamma(t)}\kappa \d s/\int_{\Gamma(t)}1 \d s$ is the average curvature and $\mathbf{n}$ is the outward unit normal to $\Gamma$. Here, we use the sign convention that a unit circle has a positive constant curvature.

By representing the curves $\Gamma(t)$ as a parametrization $\mathbf{X}(\cdot,t):\mathbb{I}\rightarrow \R^2$, where $\mathbb{I}:=\mathbb{R}/\mathbb{Z}$ is the ``periodic" interval $[0, 1]$, Barrett, Garcke and N\"urnberg \cite{BGN07A,BGN20}  creatively reformulated the above equations \eqref{Geometric equation} into the following coupled forms:
\begin{equation}\label{Coupled equation}
	\begin{split}
		\p_t\mathbf{X}\cdot \mathbf{n} &= \begin{cases}
			-\kappa, &\quad \text{CSF},\\
			-\kappa+\l<\kappa \r>,&\quad \text{AP-CSF},\\
			\p_{ss}\kappa, &\quad \text{SDF},
		\end{cases}  \\
	    \kappa \bn &= -\p_{ss}\bX.
	\end{split}
\end{equation}
Based on the above equations and the corresponding weak formulations, a series of numerical schemes (the so-called BGN schemes) were proposed for solving different geometric flows, such as mean curvature flow and surface diffusion~\cite{BGN07A,BGN07B}, Willmore flow~\cite{BGN08C}, anisotropic geometric flow~\cite{Bao-Jiang-Li}, solid-state dewetting~\cite{Zhao-Jiang-Wang-Bao,Zhao-Jiang-Bao} and geometric flow for surface evolution \cite{BGN08B}. Recently, based on the BGN formulation \eqref{Coupled equation}, structure-preserving schemes have been proposed for axisymmetric geometric equations \cite{Bao-Garcke-Nurnberg-Zhao} and surface diffusion \cite{Bao-Zhao, Bao-Jiang-Li}, respectively. In practical simulations, ample numerical results have demonstrated the high performance of
the BGN scheme, due to inheriting the variational structure of the original problem and introducing an appropriate tangential velocity to help mesh points maintain a good distribution. However, for the original BGN scheme, because its formal truncation error is $\mathcal{O}(\tau)$, where $\tau$ is the time step size, the temporal convergence order of the scheme is limited to the first-order. This has been confirmed by extensive numerical experiments~\cite{Bao-Zhao, BGN07A,BGN07B, Zhao-Jiang-Wang-Bao}. Therefore, how to design a temporal high-order scheme which is based on the BGN formulation \eqref{Coupled equation} is challenging and still open. It is also worth noting that rigorous numerical analysis for BGN schemes remains an open problem~\cite{BGN20}.

In this paper, based on the BGN formulation \eqref{Coupled equation}, we propose a novel temporal second-order parametric finite element method for solving geometric flows of curves, i.e., CSF, AP-CSF and SDF. Specifically, to discretize the same continuous-in-time semi-discrete formulation as the classical BGN scheme \cite{BGN07A}, we begin by fixing the unit normal as that on the current curve $\Gamma^m$ and then discretize other terms using the Crank-Nicolson leap-frog scheme~\cite{Hurl16}.
The resulting scheme is a second-order semi-implicit scheme, which only requires solving a system of linear algebraic equations at each time step. Furthermore, the well-posedness 
of the fully discrete scheme can be established under suitable assumption conditions.  Numerical results have demonstrated that the proposed scheme achieves second-order accuracy in time, as measured by the shape metrics, outperforming the classical BGN scheme in terms of accuracy and efficiency.

It is worth mentioning that there exist several temporal higher-order numerical schemes based on other formulations which have been proposed for simulating geometric flows. For the specific case of curve-shortening flow, a Crank-Nicolson-type scheme combined with tangential redistribution \cite{Balazovjech-Mikula} and an adaptive moving mesh method \cite{Mackenzie-Nolan-Rowlatt-Insall} have been developed. Both of the schemes are convergent quadratically in time and fully implicit, requiring to solve a system of nonlinear equations at each time step. Recently, an evolving surface finite element method together with linearly implicit backward difference formulae for time integration for simulating the mean curvature flow has been proposed in \cite{KLL2019,KLL2020}. In comparison to these existing approaches, our newly proposed scheme is based on the BGN formulation \eqref{Coupled equation}, then it inherits the variational structure of the original geometric flows, and has very good property with respect to mesh distribution. The new scheme exhibits comparable computational cost to the classical BGN scheme while surpassing it in terms of accuracy. Furthermore, it can be extended easily to other geometric flows with applications to various fields.

The main reason why we have successfully proposed a temporal high-order, BGN-based parametric finite element method for solving geometric flows lies in the following two key points: (1). we choose an appropriate metric (i.e., shape metrics) to measure numerical errors of the proposed schemes; (2). we use the classical first-order BGN scheme as ``a good partner'' of the proposed scheme to help mesh points maintain a good distribution without sacrificing the accuracy.

How to measure the errors of numerical solutions for geometric flows is an important issue. A natural approach is to use classical Sobolev norms, such as $L^2$-norm, $L^\infty$-norm or $H^1$-norm, which are widely used in the numerical analysis for geometric flows \cite{Dziuk1994,Dziuk1999,KLL2019,KLL2020}. However, when it comes to numerical schemes that involve in tangential movements, these function norms may not be suitable for quantifying the differences between two curves/surfaces. To address this issue, we consider an alternative approach using shape metrics, such as manifold distance (as used in \cite{Bao-Zhao,Zhao-Jiang-Bao2021}) and Hausdorff distance \cite{Bai2011}. These metrics provide a measure of how similar or different two curves/surfaces are in terms of their shape characteristics. Extensive numerical experiments have been conducted, and the results demonstrate that our proposed scheme achieves second-order accuracy when measured using shape metrics.

On the other hand, the quality of mesh distribution is always a major concern when simulating geometric flows using parametric finite element methods. It is important to note that the original flow \eqref{Geometric equation} requires the curve to evolve only in the normal direction, thus the numerical methods based on \eqref{Geometric equation} which prevent tangential movement of mesh points might lead to mesh distortion or clustering during the evolution. To address this issue, various approaches have been proposed in the literature to maintain good mesh quality, e.g., artificial mesh regularization method \cite{Bansch-Morin-Nochetto}, reparametrization by introducing a tangential velocity \cite{Deckelnick-Dziuk,Mikula-Sevcovic,Elliott-Fritz,Kimura,Mikula-Sevcovic2004}. On the contrary, the BGN formulation \eqref{Coupled equation} does not enforce any condition on the tangential velocity, which allows for an intrinsic tangential motion of mesh points, as demonstrated by the standard BGN scheme \cite{BGN07A,BGN07B} constructed based on this formulation \eqref{Coupled equation}. Though the semi-discrete scheme of \eqref{Coupled equation}, where only spatial discretization is performed, results in precise equidistribution of mesh points, our proposed fully discrete second-order BGN-based scheme exhibits oscillations in terms of mesh ratio and other geometric quantities, which may lead to instability in certain situations. To address this issue, we implement two classical first-order BGN schemes as mesh regularization procedures to enhance the quality of the mesh. More specifically, (1). we utilize the classical semi-implicit BGN scheme when poorly distributed polygonal approximations are detected. Extensive numerical experiments have shown that this approach improves the stability of the new scheme and significantly enhances the mesh quality. Importantly, numerous numerical experiments  have also demonstrated that this mesh regularization only occurs infrequently throughout the evolution process, ensuring that the temporal second-order accuracy of the proposed scheme remains uncompromised; (2). after solving the BGN2 scheme at each time step, we employ the fully-implicit BGN scheme for the trivial flow in order to achieve mesh equidistribution. Although this mesh regularization may increase the computational cost, the unaffected temporal second-order accuracy ensures that our newly proposed scheme remains more efficient than classical BGN schemes. More importantly, this mesh regularization allows for the establishment of unconditional energy stability.

The remaining of the paper is organized as follows. In Section 2, taking CSF as an example, we begin by recalling the standard BGN scheme, and then propose a second-order in time, BGN-based parametric finite element method for solving CSF. Two mesh regularization procedures are proposed to ensure the good mesh quality during the evolution. Section 3 is devoted to explaining the importance of using shape metrics, such as manifold distance and Hausdorff distance, to accurately measure the errors of two curves.  We extend the proposed second-order scheme to other geometric flows such as AP-CSF and the fourth-order flow SDF in Section 4. Extensive numerical results are provided to demonstrate the accuracy and efficiency of the proposed schemes in Section 5. Finally, we draw some conclusions in Section 6.

\section{For curve shortening flow (CSF)}

In this section, we propose a parametric finite element method with second-order temporal accuracy for numerically solving the CSF. The same idea can be easily extended to other geometric flows (cf. Section 4). To provide a comprehensive understanding, we first review a classical first-order BGN scheme proposed by Barrett, Garcke and N\"urnberg  \cite{BGN07A,BGN07B,BGN20}.

\subsection{Derivation of the classical BGN scheme}
To begin with, we rewrite the CSF into the following formulation as presented in Eqs.~\eqref{Coupled equation}:
\begin{equation}\label{CSF:Coupled equation}
\begin{split}
	\p_t \mathbf{X}\cdot \mathbf{n} &=-\kappa,\\
  	\kappa \mathbf{n}&=-\p_{ss}\mathbf{X}.
\end{split}
\end{equation}
We introduce the following finite element approximation. Let $\mathbb{I}=[0,1]= \bigcup_{j=1}^N I_j$, $N\ge 3$, be a decomposition of $\mathbb{I}$ into intervals given by the nodes red$\rho_j$, $I_j=[\rho_{j-1},\rho_j]$. Let $h=\max\limits_{1\le j\le N}
|\rho_j-\rho_{j-1}|$ be the maximal length of a grid element. Define the linear finite element space as
\[
V^h:=\{u\in C(\mathbb{I}): u|_{I_j} \,\,\, \mathrm{is\,\,\,linear,\,\,\,} \forall j=1,2,\ldots,N;\quad u(\rho_0)=u(\rho_N) \}\subseteq H^1(\mathbb{I}).
\]
The mass lumped inner product $(\cdot,\cdot)_{\Gamma^h}^h$ over the polygonal curve $\Gamma^h$, which is an approximation of $(\cdot,\cdot)_{\Gamma^h}$ by using the composite trapezoidal rule, is defined as
\[
(u,v)_{\Gamma^h}^h:=\frac{1}{2}\sum_{j=1}^N|\bX^h(\rho_j,t)-\bX^h(\rho_{j-1},t)|\l[(u\cdot v)(\rho_j^-)+(u\cdot v)(\rho_{j-1}^+) \r],
\]
where $u, v$ are two scalar/vector piecewise continuous functions with possible jumps at the nodes $\{\rho_j\}_{j=1}^N$,
and $u(\rho_j^{\pm})=\lim\limits_{\rho\rightarrow \rho_j^{\pm}}u(\rho)$.

Subsequently, the semi-discrete scheme of the formulation \eqref{CSF:Coupled equation} is as follows: given initial polygon $\Gamma^h(0)$ with vertices lying on the initial curve $\Gamma(0)$ clockwise, parametrized by $\bX^h(\cdot,0)\in [V^h]^2$,
find $(\bX^h(\cdot,t),\kappa^h(\cdot,t))\in [V^h]^2\times V^h$ such that
\begin{equation}\label{CSF:Semi-discrete}
	\begin{cases}
			\l(\p_t\mathbf{X}^h\cdot  \mathbf{n}^h,\varphi^h \r)_{\Gamma^h}^h+\l(  \kappa^h,\varphi^h\r)^h_{\Gamma^h}=0,\quad \forall\ \varphi^h\in V^h,\\
			\l(\kappa^h,\mathbf{n}^h\cdot \bm{\omega}^h\r)^h_{\Gamma^h}-\l(\p_s \mathbf{X}^h,\p_s\bm{\omega}^h \r)_{\Gamma^h}=0,\quad \forall\ \bm{\omega}^h\in  [V^h]^2,
		\end{cases}
\end{equation}
where we always integrate over the current curve $\Gamma^h$ described by $\mathbf{X}^h$, the outward unit normal $\bn^h$ is a piecewise constant vector given by
\[\bn^h|_{I_j}=-\frac{\mathbf{h}_j^\perp}{|\mathbf{h}_j|}, \quad \mathbf{h}_j=\bX^h(\rho_j,t)-\bX^h(\rho_{j-1},t),\quad j=1,\ldots, N,\]
with  $\cdot^\perp$ denoting clockwise rotation by $\frac{\pi}{2}$, and
the partial derivative $\p_s$ is defined piecewisely over each side of the polygon
$\p_s f|_{I_j}=\frac{\p_\rho f}{|\p_\rho \mathbf{X}^h|}|_{I_j}=\frac{(\rho_j-\rho_{j-1})\p_\rho f|_{I_j}}{|\mathbf{h}_j|}$. It was shown that the scheme \eqref{CSF:Semi-discrete} will always equidistribute the vertices along $\Gamma^h$ for $t>0$ if they are not locally parallel (see Remark 2.4 in \cite{BGN07A}).

For a full discretization, we fix $\tau>0$ as a uniform time step size for simplicity, and let $\bX^m\in [V^h]^2$ and $\Gamma^m$ be the approximations of $\bX(\cdot,t_m)$ and $\Gamma(t_m)$, respectively, for $m=0,1,2,\ldots$, where $t_m:=m\tau$. We define $\mathbf{h}_j^m:=\bX^m(\rho_j)-\bX^m(\rho_{j-1})$ and assume $|\mathbf{h}_j^m|>0$ for $j=1,\ldots,N$, $\forall\ m>0$. The discrete unit normal vector $\bn^m$, the discrete inner product $(\cdot,\cdot)^h_{\Gamma^m}$ and the discrete operator $\p_s$ are defined similarly as in the semi-discrete case.
Barrett, Garcke and N\"urnberg used a formal first-order approximation \cite{BGN07A,BGN07B} to replace the velocity $\p_t \bX$, $\kappa$ and $\p_s\bX$ by
\begin{equation*}
\begin{split}
     \p_t \bX(\cdot, t_m)&= \frac{\mathbf{X}(\cdot,t_{m+1})-\mathbf{X}(\cdot, t_m)}{\tau}+\mathcal{O}(\tau), \\
	\kappa(\cdot,t_m)&=\kappa(\cdot,t_{m+1})+\mathcal{O}(\tau),  \\
	\p_s\bX(\cdot,t_m)&= \p_s \mathbf{X}(\cdot, t_{m+1})+\mathcal{O}(\tau),
\end{split}
\end{equation*}
and the fully discrete semi-implicit BGN scheme (denoted as BGN1 scheme) reads as:

(\textbf{BGN1, First-order in time BGN scheme for CSF}): For $m\ge 0$, find $\mathbf{X}^{m+1}\in [V^h]^2$ and $\kappa^{m+1}\in V^h$ such that
\begin{equation}\label{CSF:BGN1}
		\begin{cases}
			\l(\frac{\mathbf{X}^{m+1}-\mathbf{X}^m}{\tau},\varphi^h \mathbf{n}^m \r)^h_{\Gamma^m}+\l(  \kappa^{m+1},\varphi^h \r)_{\Gamma^m}^h=0,\quad \forall\ \varphi^h\in V^h,\\	\l(\kappa^{m+1},\mathbf{n}^m\cdot \bm{\omega}^h\r)_{\Gamma^m}^h-\l(\p_s \mathbf{X}^{m+1},\p_s\bm{\omega}^h\r)_{\Gamma^m}=0,\quad \forall\ \bm{\omega}^h\in  [V^h]^2.
		\end{cases}
	\end{equation}
The well-posedness and energy stability were established under some mild conditions. In practice, numerous numerical results show that the BGN1 scheme \eqref{CSF:BGN1} converges quadratically in space \cite{BGN07B} and linearly in time (cf.  Fig. \ref{Fig:CSF_EOC1} in Section \ref{sec:order 2, illu}).

\subsection{A second-order in time, BGN-based scheme}
\label{sec:order 2}
Instead of using the first-order Euler method, we apply the Crank-Nicolson leap-frog time stepping discretization in \eqref{CSF:Semi-discrete} based on the following simple calculation
\be\label{app}
\begin{split}
	\p_t\bX(\cdot,t_m)&= \frac{\mathbf{X}(\cdot, t_{m+1})-\mathbf{X}(\cdot, t_{m-1})}{2\tau}+\mathcal{O}(\tau^2),\\
	\kappa(\cdot,t_m)&= \frac{\kappa(\cdot, t_{m+1})+\kappa(\cdot, t_{m-1})}{2}+\mathcal{O}(\tau^2),\\
	\p_s\bX(\cdot,t_m)&= \frac{\p_s \mathbf{X}(\cdot, t_{m+1})+\p_s \mathbf{X}(\cdot, t_{m-1})}{2}+\mathcal{O}(\tau^2),
\end{split}
\ee
then the corresponding second-order scheme (denoted as BGN2 scheme) is as follows:

(\textbf{BGN2, Second-order in time BGN-based scheme for CSF}):~ For $\bX^0 \in [V^h]^2$, $\kappa^0\in  V^h$ and $(\bX^1,\kappa^1)\in [V^h]^2\times V^h$ which are the appropriate approximations at the time levels $t_0=0$ and $t_1=\tau$, respectively, find $\mathbf{X}^{m+1}\in [V^h]^2$ and $\kappa^{m+1}\in V^h$ for $m\ge 1$ such that
\begin{equation}\label{CSF:BGN2}
		\begin{cases}
	\l(\frac{\mathbf{X}^{m+1}-\mathbf{X}^{m-1}}{2\tau},\varphi^h \mathbf{n}^m \r)^h_{\Gamma^m}+\l(  \frac{\kappa^{m+1}+\kappa^{m-1}}{2},\varphi^h \r)_{\Gamma^m}^h=0,\\
	\vspace{-3mm}\\		\l(\frac{\kappa^{m+1}+\kappa^{m-1}}{2},\mathbf{n}^m\cdot \bm{\omega}^h\r)_{\Gamma^m}^h-\l(\frac{\p_s \mathbf{X}^{m+1}+\p_s \mathbf{X}^{m-1}}{2},\p_s\bm{\omega}^h\r)_{\Gamma^m}=0,
		\end{cases}
	\end{equation}
for all $(\varphi^h, \bm{\omega}^h)\in V^h\times [V^h]^2$.
The scheme \eqref{CSF:BGN2} is semi-implicit and the computational cost is comparable to that of the BGN1 scheme \eqref{CSF:BGN1}. Moreover, as a temporal discretization of the semi-discrete version \eqref{CSF:Semi-discrete}, it can be easily derived from \eqref{app} that the truncation error is of order $\mathcal{O}(\tau^2)$.

\begin{remark}
To begin the BGN2 scheme \eqref{CSF:BGN2}, we need to first prepare the data $\kappa^0$ and
 $(\bX^1,\kappa^1)$.  In practical simulations, this can be easily achieved without sacrificing the accuracy of the scheme by utilizing the standard BGN1 scheme \eqref{CSF:BGN1} to get $(\bX^1,\kappa^1)$, and the following formula of discrete curvature was proposed in \cite[Page 461]{BGN07A} to prepare $\kappa^0$ (note the the sign convention of the curvature is opposite to \cite{BGN07A})
 \begin{equation}\label{kappa formula}
 	\kappa^0=(N_0^\top N_0)^{-1}N_0^\top A_0 \mathbf{X}^0,
 \end{equation}
where $N_0$ is a $2N\times N$ matrix, $\mathbf{X}^0$ is a $2N\times 1$ vector and $A_0$ is a $2N\times 2N$ matrix given by
\begin{align*}
&N_0 = \begin{pmatrix}
	(\varphi_i,(\bn^0)^{[1]} \varphi_j)_{\Gamma^0}^h\\
	(\varphi_i,(\bn^0)^{[2]} \varphi_j)_{\Gamma^0}^h
\end{pmatrix}, \quad \mathbf{X}^0=\begin{pmatrix}
	\mathbf{x}^0\\
	\mathbf{y}^0
\end{pmatrix},\\
&A_0=\begin{pmatrix}
	(\p_s\varphi_i,\p_s\varphi_j)_{\Gamma^0} & 0\\
	0 & (\p_s\varphi_i,\p_s\varphi_j)_{\Gamma^0}
\end{pmatrix},
\end{align*}
where $\varphi_i, 1\le i\le N$ are the standard Lagrange basis over $\mathbb{I}$, and  $\mathbf{a}^{[1]},\mathbf{a}^{[2]}$ are the first and second component of vector $\mathbf{a}\in \R^2$, and $\mathbf{x}_j^0=(\bX^0)^{[1]}(\rho_j)$, $\mathbf{y}_j^0=(\bX^0)^{[2]}(\rho_j)$ for $j=1,\ldots, N$. Note that this formula can be derived by solving the finite element approximation of the equation $\kappa \bn=-\p_{ss}\bX$ and using the least square method. We can summarize the process as Algorithm \ref{CSF:BGN initial data 1}, which outlines the steps to prepare the required data $\kappa^0$ and $(\bX^1,\kappa^1)$. Once we have obtained these data, we can directly apply the BGN2 scheme \eqref{CSF:BGN2} to calculate $(\bX^m,\kappa^m)$, for $m\ge 2$.
 \end{remark}

\begin{algorithm}
\textbf{(Preparation for the initial data of BGN2 for CSF)}
\label{CSF:BGN initial data 1}

$\bm{Step~0.}$ Given the initial curve $\Gamma(0)$, the number of grid points $N$ and the time step size $\tau$. We choose the polygon $\Gamma^0$ with $N$ vertices lying on $\Gamma(0)$ such that $\Gamma^0$ is (almost) equidistributed, i.e., each side of the polygon is (nearly) equal in length. We parameterize $\Gamma^0$ with $\bX^0\in [V^h]^2$ and the grid points $\rho_j$ can be determined correspondingly.

$\bm{Step~1.}$ Using $\bX^0$ as the input, we compute $\kappa^0$ using the discrete curvature formula \eqref{kappa formula}.

 $\bm{Step~2.}$ Using $\bX^0$ as the input, we obtain $(\bX^1,\kappa^1)$ by solving the BGN1 scheme \eqref{CSF:BGN1} for one time step.

\end{algorithm}

 \begin{remark}\label{iniitalp}
 When dealing with an initial curve which is not regular, an alternative approach for initialization is to solve the BGN1 scheme twice and start the BGN2 scheme from $m=2$. Specifically, for given $\bX^0$, we can compute $(\bX^1, \kappa^1)$
 and $(\bX^2, \kappa^2)$, which are the appropriate approximations at time levels $t_1=\tau$ and $t_2=2\tau$, by solving the BGN1 scheme \eqref{CSF:BGN1} twice. These approximations can be used as initial values to implement the BGN2 scheme \eqref{CSF:BGN1} for $m\ge 2$. For the superiority of this approach, see Fig.~\ref{Fig:flower_geo_MRandnoMRb} in Section 5.3.	
 \end{remark}

Similar to the BGN1 scheme \eqref{CSF:BGN1}, we can show the well-posedness of the BGN2 scheme \eqref{CSF:BGN2} under some mild conditions as follows.
	
\smallbreak
	
\begin{thm}[Well-posedness]\label{CSF:Well-posedness} For $m\ge 0$,  we assume that the following two conditions are satisfied:
	\begin{enumerate}
		\item[(1)] There exist at least two vectors in $\{\mathbf{h}_j^m\}_{j=1}^{N}$ which are not parallel, i.e.,
		\[
		\mathrm{dim}\l( \mathrm{Span}\l\{\mathbf{h}_j^m \r\}_{j=1}^{N}\r)=2.
		\]
		\item[(2)] No degenerate elements exist on $\Gamma^m$, i.e.,
		\[
		\min_{1\le j\le N}| \mathbf{h}_j^m|>0.
		\]
	\end{enumerate}
	Then the full discretization \eqref{CSF:BGN2} is well-posed, i.e., there exists a unique solution $(\mathbf{X}^{m+1},\kappa^{m+1})\in [V^h]^2\times V^h$ of \eqref{CSF:BGN2}.
\end{thm}

\begin{proof}
	It suffices to prove the following algebraic system for $(\bX,\kappa)\in [V^h]^2\times V^h$ has only zero solution,
	\begin{equation*}
		\begin{cases}
			\l(\frac{\mathbf{X}}{\tau},\varphi^h \mathbf{n}^m \r)^h_{\Gamma^m}+\l( \kappa ,\varphi^h \r)_{\Gamma^m}^h=0,\quad \forall\ \varphi^h\in V^h,\\
   \l(\kappa,\mathbf{n}^m\cdot \bm{\omega}^h\r)_{\Gamma^m}^h-\l(\p_s \mathbf{X},\p_s\bm{\omega}^h\r)_{\Gamma^m}=0,\quad \forall\ \bm{\omega}^h\in  [V^h]^2.
		\end{cases}
	\end{equation*}
	Indeed, the stiffness matrix is exactly the same as the standard BGN1 scheme \eqref{CSF:BGN1} and thus the same argument in \cite[Theorem 2.9]{BGN07B} yields the conclusion under the assumptions (1) and (2).
\end{proof}

	\smallbreak

\smallbreak

\subsection{Mesh regularization by semi-implicit BGN1 scheme}
As was mentioned earlier, the semi-discrete scheme \eqref{CSF:Semi-discrete} possesses the mesh equidistribution property \cite[Theorem 79]{BGN20}. In practice, the fully-discrete BGN1 scheme \eqref{CSF:BGN1} can maintain the asymptotic long-time mesh equidistribution property. However, the BGN2 scheme~\eqref{CSF:BGN2} may have oscillating mesh ratio due to the structure of two-step method, which can potentially amplify the mesh ratio and cause mesh distortion or clustering during the evolution, especially for some initial curves which are not so regular, e.g., a `flower'  curve (see the second row of Fig. \ref{Fig:flower_evo_MRandnoMR}). Therefore, a mesh regularization procedure is necessary in real simulations to help the mesh maintain a good distribution property during the evolution, when the mesh ratio exceeds a given threshold value. Inspired by the good mesh distribution property of the BGN1 scheme, we utilize the BGN1 scheme as the mesh regularization technique. In the following, we denote $n_{\rm{MR}}$ as the threshold value chosen initially. If the mesh ratio $\Psi^{m}>n_{\rm{MR}}$, then we use the mesh regularization procedure to improve the mesh distribution.  We present a summary of the complete algorithm of BGN2 scheme for solving the CSF in Algorithm \ref{Full algorithm}.

\begin{algorithm}
\textbf{(\textbf{BGN2 scheme for CSF})}
\label{Full algorithm}

$\bm{Step~0.}$  Given the initial curve $\Gamma(0)$, and $N,T,n_{\rm{MR}}$, $\tau$, compute $\bX^0$ as in \emph{Step~0} in Algorithm \ref{CSF:BGN initial data 1}.

$\bm{Step~1.}$ Using $\bX^0$ as the input, we compute $\kappa^0$ using the discrete curvature formula \eqref{kappa formula} and solve $(\bX^1, \kappa^1)$ via the BGN1 scheme \eqref{CSF:BGN1}. Set $m=1$.

$\bm{Step~2.}$ Calculate the mesh ratio $\Psi^m$ of $\bX^m$, $m\ge 1$.

$\bm{Step~3.}$ If the mesh ratio $\Psi^m>n_{\rm{MR}}$, then replace $(\bX^m,\kappa^m)$ with the solution of the BGN1 scheme \eqref{CSF:BGN1} with $\bX^{m-1}$ as the input by one run; otherwise, skip this step.

$\bm{Step~4.}$ Use the BGN2 scheme \eqref{CSF:BGN2} to obtain $(\bX^{m+1},\kappa^{m+1})$.

$\bm{Step~5.}$ Update $m=m+1$. If $m< T/\tau$, then go back to $\textbf{Step~2}$; otherwise, stop the algorithm and output the data.

\end{algorithm}

\smallskip

As shown in $\emph{\textbf{Step~3}}$  of Algorithm \ref{Full algorithm},  if the mesh ratio $\Psi^{m}>n_{\rm{MR}}$, we replace $(\bX^m,\kappa^m)$ with the solution of the BGN1 scheme~\eqref{CSF:BGN1} with $\bX^{m-1}$ as the input by one run,
to help us realize the mesh regularization. Extensive numerical experiments suggest that the mesh regularization procedure is very effective, and the mesh ratio decreases immediately to a small value after this procedure (cf. Fig.~\ref{Fig:tube_evo_and_Geo}(d) in Section 5). The BGN2 scheme with the aid of the BGN1 scheme as the mesh regularization is very efficient and
stable in practical simulations. The reason comes from that the BGN1 scheme \eqref{CSF:BGN1} can intrinsically lead to a good mesh distribution property, which was explained in \cite{BGN07A,BGN20}, but a more convincing explanation needs further rigorous numerical analysis for the scheme.

One concern that may arise is whether the BGN2 scheme with necessary mesh regularization can still achieve second-order accuracy, considering that the BGN1 scheme is only first-order accurate. It is important to note that for certain smooth initial curves, such as elliptic curves, the mesh regularization procedure is never required during the evolution. In such cases, the numerical evolution remains remarkably stable and the mesh ratio remains bounded.  While for certain special initial curves, like a `flower'  curve or a `tube' curve, the mesh regularization procedure may be needed only a few times  (cf. Section \ref{sec:long time, illu}). Nevertheless, this does not compromise the temporal second-order accuracy of the BGN2 scheme~\eqref{CSF:BGN2}.

\subsection{Mesh regularization by implicit equi-BGN1 scheme for trivial flow}
In the following, we recall a fully-implicit scheme for CSF \cite{BGN11} which intrinsically equidistributes the vertices along the curve at each time step.

(\textbf{equi-BGN1, First-order in time equidistribution BGN scheme for CSF}): For $m\ge 0$, find $\mathbf{X}^{m+1}\in [V^h]^2$ and $\kappa^{m+1}\in V^h$ such that
\begin{equation}\label{CSF:equi-BGN1}
\begin{cases}
\l(\frac{\mathbf{X}^{m+1}-\mathbf{X}^m}{\tau},\varphi^h \mathbf{n}^{m+1} \r)^h_{\Gamma^{m+1}}+\l(\kappa^{m+1},\varphi^h \r)_{\Gamma^{m+1}}^h=0,\,\,\,\, \forall\,\varphi^h\in V^h,\\	
\l(\kappa^{m+1},\mathbf{n}^{m+1}\cdot \bm{\omega}^h\r)_{\Gamma^{m+1}}^h-\l(\p_s \mathbf{X}^{m+1}, \p_s\bm{\omega}^h\r)_{\Gamma^{m+1}}=0,\,\,\,\, \forall\,\bm{\omega}^h\in  [V^h]^2.
		\end{cases}
	\end{equation}
It has been shown in \cite{BGN11} that
	\begin{equation}\label{Lemma:property1}
		|\mathbf{h}_{j}^{m+1}|=|\mathbf{h}_{j-1}^{m+1}|,\quad \mathrm{if}\quad \mathbf{h}_{j}^{m+1}\nparallel \mathbf{h}_{j-1}^{m+1},\qquad j=1,\ldots,N.
	\end{equation}
Moreover, the stability estimate holds
	\begin{equation}\label{Lemma:property2}
		L^{m+1}+\tau\l(\kappa^{m+1},\kappa^{m+1}\r)_{\Gamma^{m+1}}^h\le L^{m},
	\end{equation}
where $L^m$ represents the length of the polygon $\Gamma^{m}$.

Inspired by the equidistribution property of the fully implicit scheme \eqref{CSF:equi-BGN1}, we propose to implement the mesh regularization using the equi-BGN1 scheme for the trivial flow
\[
	\cV=0,
\]
which can distribute mesh points equally and retain the shape of the curve in the continuous level. More specifically, with $\mathbf{X}^{m}\in [V^h]^2$, find $\widetilde{\mathbf{X}}^{m}\in [V^h]^2$ and $\widetilde{\kappa}^{m}\in V^h$ such that
\begin{equation}\label{tri:equi-BGN1}
\begin{cases}
\l(\widetilde{\mathbf{X}}^{m}-\mathbf{X}^{m}, \varphi^h \widetilde{\mathbf{n}}^{m} \r)^h_{\widetilde{\Gamma}^{m}}=0,\quad \forall\,\varphi^h\in V^h,\\	
\l(\widetilde{\kappa}^{m}, \widetilde{\mathbf{n}}^{m}\cdot \bm{\omega}^h\r)_{\widetilde{\Gamma}^{m}}^h-\l(\p_s \widetilde{\mathbf{X}}^{m}, \p_s\bm{\omega}^h\r)_{\widetilde{\Gamma}^{m}}=0,\quad \forall\,\bm{\omega}^h\in  [V^h]^2.
		\end{cases}
	\end{equation}
Similar to \eqref{CSF:equi-BGN1}, it can be rigorously proved that the vertices of $\widetilde{\Gamma}^m$ are evenly distributed and the perimeter does not increase, i.e.,
\be\label{equid}
|\widetilde{\mathbf{h}}_{j}^{m}|=|\widetilde{\mathbf{h}}_{j-1}^{m}|, \quad j=1,\ldots, N;\qquad \widetilde{L}^m\le L^m.
\ee

Now, we present a summary of the entire algorithm for the equi-BGN2 scheme for solving the CSF. This scheme can be regarded as a variant of scheme \eqref{CSF:BGN2}, where $(\bX^{m-1},\kappa^{m-1})$ and $\Gamma^{m}$ are replaced by their mesh regularized approximations $(\widetilde{\bX}^{m-1}, \widetilde{\kappa}^{m-1})$ and $\widetilde{\Gamma}^{m}$, respectively.

\begin{algorithm}{(\textbf{equi-BGN2 scheme for CSF}})
\label{Full algorithm 2}

$\bm{Step~0.}$  Given the initial curve $\Gamma(0)$, and $N,T$, $\tau$, compute $\bX^0$ as in \emph{Step~0} in Algorithm \ref{CSF:BGN initial data 1}. Use equi-BGN1 scheme \eqref{tri:equi-BGN1} to obtain the equidistributed polygon $\widetilde{\bX}^0$ and $\widetilde{\kappa}^0$.

$\bm{Step~1.}$ Using $\widetilde{\bX}^0$ as the input, we obtain $(\widetilde{\bX}^1, \widetilde{\kappa}^1)$ by solving the equi-BGN1 scheme \eqref{CSF:equi-BGN1} with time step $\tau$. Set $m=1$. 

$\bm{Step~2.}$ Solve the BGN2 scheme \eqref{CSF:BGN2} with $(\widetilde{\bX}^{m-1}, \widetilde{\kappa}^{m-1})$ and $\widetilde{\Gamma}^m$ to obtain $(\bX^{m+1},\kappa^{m+1})$.

$\bm{Step~3.}$ Update $m=m+1$. Apply the equi-BGN1 scheme \eqref{tri:equi-BGN1} to obtain the mesh-regularized approximation  $\widetilde{\bX}^{m}$ and $\widetilde{\kappa}^{m}$.

$\bm{Step~4.}$ If $m< T/\tau$, then go back to $\textbf{Step~2}$; otherwise, stop the algorithm and output the data $\widetilde{\bX}^{m}$ as an approximation solution at time $t_m$.

\end{algorithm}

\smallskip

Indeed, the solution $\widetilde{\bX}^{m}$ not only equidistributes the vertices at each time level, but also is unconditionally stable in the following sense.

\begin{thm}[Interlaced energy stability]\label{CSF:Interlaced energy stability} Let $\widetilde{\Gamma}^{m}=\widetilde{\bX}^{m}(\mathbb{I})$ be the solution of Algorithm \ref{Full algorithm 2}. Then for any $\tau>0$ and $m\ge 1$, the energy stability holds
\begin{equation}\label{CSF:energy stability}
	\widetilde{L}^{m+1}\le \widetilde{L}^{m-1},
\end{equation}
where $\widetilde{L}^{m}:=\sum\limits_{j=1}^N |\widetilde{\mathbf{h}}^{m}_j|
$ is the perimeter of $\widetilde{\Gamma}^{m}$. In particular, we have
\begin{equation}\label{CSF:energy stability 2}
	\widetilde{L}^{m}\le L^0,\quad \forall\ m\ge 1,
\end{equation}
where $L^0$ is the perimeter of the initial polygon $\Gamma^0$.
\end{thm}

\begin{proof}
Recalling $\bm{Step~2}$, taking  $\bm{\omega}^h=\frac{\bX^{m+1}-\widetilde{\bX}^{m-1}}{2\tau}$ and $\varphi^h=\frac{\kappa^{m+1}+\widetilde{\kappa}^{m-1}}{2}$ in equation  \eqref{CSF:BGN2}, we get
	\begin{align}
		&\Big(  \frac{\kappa^{m+1}+\widetilde{\kappa}^{m-1}}{2},
\frac{\kappa^{m+1}+\widetilde{\kappa}^{m-1}}{2}\Big)_{\widetilde{\Gamma}^m}^h
\notag \\	&=-\Big(\frac{\mathbf{X}^{m+1}-\widetilde{\mathbf{X}}^{m-1}}{2\tau},\l(\frac{\kappa^{m+1}+\widetilde{\kappa}^{m-1}}{2}\r) \widetilde{\mathbf{n}}^m \Big)^h_{\widetilde{\Gamma}^m}\notag\\
		&=-\Big(\frac{\p_s \mathbf{X}^{m+1}+\p_s \widetilde{\mathbf{X}}^{m-1}}{2},\frac{\p_s \mathbf{X}^{m+1}-\p_s \widetilde{\mathbf{X}}^{m-1}}{2\tau}\Big)_{\widetilde{\Gamma}^m}\notag\\
		&=-\frac{1}{4\tau}\Big(\l(\p_s \mathbf{X}^{m+1},\p_s \mathbf{X}^{m+1}\r)_{\widetilde{\Gamma}^m}-\Big(\p_s \widetilde{\mathbf{X}}^{m-1},\p_s \widetilde{\mathbf{X}}^{m-1}\Big)_{\widetilde{\Gamma}^m} \Big).\label{CSF:ES proof 1}
	\end{align}
Noticing \eqref{equid}, we denote $|\widetilde{\mathbf{h}}^{m}|=\frac{\widetilde{L}^{m}}{N}$ by the length of each edge of the polygon $\widetilde{\Gamma}^m$, then, we have
	\begin{align*}
		\l(\p_s \mathbf{X}^{m+1},\p_s \mathbf{X}^{m+1}\r)_{\widetilde{\Gamma}^m}
		&=\sum_{j=1}^N\frac{|\mathbf{h}_j^{m+1}|}{|\widetilde{\mathbf{h}}^m|}\frac{|\mathbf{h}_j^{m+1}|}{|\widetilde{\mathbf{h}}^m|}|\widetilde{\mathbf{h}}^m|=\sum_{j=1}^N\frac{|\mathbf{h}_j^{m+1}|^2}{|\widetilde{\mathbf{h}}^m|},\\
			\l(\p_s \widetilde{\mathbf{X}}^{m-1},\p_s \widetilde{\mathbf{X}}^{m-1}\r)_{\widetilde{\Gamma}^m}	&=\sum_{j=1}^N\frac{|\widetilde{\mathbf{h}}^{m-1}|}{|\widetilde{\mathbf{h}}^m|}
\frac{|\widetilde{\mathbf{h}}^{m-1}|}{|\widetilde{\mathbf{h}}^m|}|\widetilde{\mathbf{h}}^m|=\sum_{j=1}^N\frac{|\widetilde{\mathbf{h}}^{m-1}|^2}{|\widetilde{\mathbf{h}}^m|}.
	\end{align*}
Therefore, by combining with the Cauchy-Schwarz inequality, we can estimate
\begin{align}
	&\Big(\p_s \mathbf{X}^{m+1},\p_s \mathbf{X}^{m+1}\Big)_{\widetilde{\Gamma}^m}-\Big(\p_s \widetilde{\mathbf{X}}^{m-1},\p_s \widetilde{\mathbf{X}}^{m-1}\Big)_{\widetilde{\Gamma}^m}\notag\\ &=\sum\limits_{j=1}^N\frac{|\mathbf{h}_j^{m+1}|^2}{|\widetilde{\mathbf{h}}^m|}-\sum_{j=1}^N
\frac{|\widetilde{\mathbf{h}}^{m-1}|^2}{|\widetilde{\mathbf{h}}^m|}=
\frac{N}{\widetilde{L}^m}\sum\limits_{j=1}^N|\mathbf{h}_j^{m+1}|^2-
\frac{(\widetilde{L}^{m-1})^2}{\widetilde{L}^m}\notag\\
&\ge \frac{N}{\widetilde{L}^m} \Big(\sum\limits_{j=1}^N |\mathbf{h}_j^{m+1}|\Big)^2/N-  \frac{(\widetilde{L}^{m-1})^2}{\widetilde{L}^m}=\frac{(L^{m+1})^2-(\widetilde{L}^{m-1})^2}{\widetilde{L}^m}\notag \\
	&\ge \frac{(\widetilde{L}^{m+1})^2-(\widetilde{L}^{m-1})^2}{\widetilde{L}^m},
\label{CSF:ES proof 2}
\end{align}
where for the last inequality we used \eqref{equid}. Combining \eqref{CSF:ES proof 1} and \eqref{CSF:ES proof 2}, we can deduce \eqref{CSF:energy stability}.

To show \eqref{CSF:energy stability 2}, it suffices to prove
$\widetilde{L}^0\le L^0$ and $\widetilde{L}^1\le L^0$. This can be easily obtained by recalling $\bm{Step~0}$, $\bm{Step~1}$, \eqref{equid} and \eqref{Lemma:property2}.
\end{proof}
\begin{remark}
	It is also feasible to perform the mesh regularization using the semi-implicit BGN1 scheme \eqref{CSF:BGN1} for the trivial flow in $\bm{Step~3}$ of Algorithm \ref{Full algorithm 2} at each time step. While this approach can reduce the global computational costs, achieving a theoretical proof of energy stability, as demonstrated in Theorem \ref{CSF:Interlaced energy stability}, seems unattainable.
\end{remark}

In subsequent sections, we will denote \eqref{CSF:BGN1} and \eqref{CSF:equi-BGN1} by the BGN1 and equi-BGN1 scheme, respectively. We call Algorithm \ref{Full algorithm} and Algorithm \ref{Full algorithm 2} as the BGN2 and equi-BGN2 scheme, respectively.

\section{Shape metric is a better choice}
\label{sec:Different error}

As we are aware, it is an interesting and thought-provoking problem to determine how to quantify the difference between two curves in 2D or two surfaces in 3D. Given two closed curves $\Gamma_1$ and $\Gamma_2$, we assume that the two curves are parametrized by $\bX(\rho)$ and $\bY(\rho)$, respectively, over the same interval $\mathbb{I}$. Consequently, we can define the following four metrics for measurement:
\begin{itemize}
 \item (\textbf{$L^2$-error}) ~The $L^2$-norm between the parametrized functions $\bX(\rho)$ and $\bY(\rho)$ is defined in a classical way
    \[
    A(\bX,\bY) :=\|\bX(\rho)-\bY(\rho)\|_{L^2(\mathbb{I})}.
    \]

 \item (\textbf{$L^\infty$-error})~The $L^\infty$-norm between the parametrized functions  $\bX(\rho)$ and $\bY(\rho)$ is defined as
    \[
    B(\bX,\bY) :=\|\bX(\rho)-\bY(\rho)\|_{L^\infty(\mathbb{I})}.
    \]

\item (\textbf{Manifold distance}) ~The manifold distance between the curves $\Gamma_1$ and $\Gamma_2$ is defined as \cite{Zhao-Jiang-Bao2021}
	\begin{align*}
	\mathrm{M}\l(\Gamma_1,\Gamma_2\r)
	&: = |(\Omega_1\setminus\Omega_2)\cup (\Omega_2\setminus\Omega_1) | =|\Omega_1 |+|\Omega_2 |-2 |\Omega_1\cap \Omega_2 |,
\end{align*}
where $\Omega_1$ and $\Omega_2$ represent the regions enclosed by $\Gamma_1$ and $\Gamma_2$, respectively, and $|\Omega|$ denotes the area of $\Omega$.

\item(\textbf{Hausdorff distance}) ~The Hausdorff distance between the curves $\Gamma_1$ and $\Gamma_2$ is defined as \cite{Bai2011}
\[
H(\Gamma_1,\Gamma_2) = \max\{\widetilde{H}(\Gamma_1,\Gamma_2),\widetilde{H}(\Gamma_2,\Gamma_1)\},
\]
where $\widetilde{H}(\Gamma_1,\Gamma_2) = \max\limits_{a\in \Gamma_1}\min\limits_{b\in \Gamma_2}d(a,b)$, and
$d$ is the Euclidean distance.
\end{itemize}
\smallskip
\begin{remark}
The $L^2$-error and $L^\infty$-error fall within the domain of {\it{function metrics}}, which rely on the parametrization of curves.
 On the other hand, as demonstrated in \cite[Proposition 5.1]{Zhao-Jiang-Bao2021} and \cite{Bai2011},  it has been easily proven that both manifold distance and Hausdorff distance fulfill the properties of  symmetry, positivity and the triangle inequality. Therefore, they belong to the category of {\it{shape metrics}} and not influenced by the specific parametrization.
\end{remark}

\smallskip
\begin{remark}
It should be noted that the aforementioned shape metrics can be easily calculated using simple algorithms. As the numerical solutions are represented as polygons, it is very easy to calculate the area of the symmetric difference region, i.e., the manifold distance, between two polygonal curves. Additionally, a polygon-based approach proposed in the literature~\cite{Bai2011} can be employed to calculate the Hausdorff distance between planar curves.
\end{remark}

In order to test the convergence rate of numerical schemes, for example, we consider the evolution of the CSF with an initial ellipse defined by
\[
    \{(x,y)\in \mathbb{R}^2: \quad x^2+4y^2=4\}.
\]
This initial ellipse is approximated using an equidistributed polygon $\bX^0$
 with $N$ vertices. Here, we simulate the CSF by using three different numerical schemes: Dziuk's scheme \cite[Section 6]{Dziuk1994},  BGN1 scheme  and BGN2 scheme. Since the exact solution of the CSF for an elliptical curve is unknown, we first compute a reference solution $\bX_{\mathrm{ref}}$ by Dziuk's scheme (to test the convergence of Dziuk's scheme) or the BGN2 scheme (to test the convergence of BGN-type schemes) with a fine mesh and a tiny time step size, e.g., $N=10000$ and $\tau=10^{-1}*2^{-11}$. To test the temporal error, we still take a large number of grid points, e.g., $N=10000$, such that the spatial error is ignorable. The numerical error and the corresponding convergence order are then determined as follows
 \begin{equation}
 \label{eqn:errordef1}
\cE_{\mathcal{M}}:=\cE_{\tau}(T)= \mathcal{M} (\bX^k_{\tau}, \bX_{\mathrm{ref}}),
\quad \text{Order}=\log\Big(\frac{\cE_{\tau}(T)}{\cE_{\tau/2}
(T)} \Big)\Big/ \log 2,
\end{equation}
where $k=T/\tau$, and $\mathcal{M}$ represents any one of the four metrics defined above.

Tables \ref{Tab:Different_norm_ellipse_Dziuk}-\ref{Tab:Different_norm_ellipse_BGN2} display the numerical errors at time $T=0.25$ measured by the four different metrics for Dziuk's scheme \cite{Dziuk1994}, the BGN1 scheme  and the BGN2 scheme, respectively. As anticipated, we easily observe linear convergence in time for Dziuk's scheme across all four different metrics. While linear and quadratic convergence for both shape metrics (i.e., the manifold distance and Hausdorff distance) are observed for the BGN1 scheme in Table \ref{Tab:Different_norm_ellipse_BGN1} and the BGN2 scheme in Table \ref{Tab:Different_norm_ellipse_BGN2}, respectively.

\begin{table}[h!]
\def\temptablewidth{0.95\textwidth}
\caption{Numerical errors quantified  by various metrics for Dziuk's scheme \cite[Section 6]{Dziuk1994}, with the  parameters $N=10000, \tau_0=1/40$, and $T=0.25$.}
{\rule{\temptablewidth}{1pt}}
\begin{tabular*}{\temptablewidth}{@{\extracolsep{\fill}}llllll}
\text{Errors}
 & $\tau=\tau_0$ & $\tau_0/2$ &$\tau_0/2^2$  &$\tau_0/2^3$  \\\hline
$ L^2$-norm  & 1.17E-2 & 6.31E-3 &3.26E-3 &1.62E-3  \\
    Order &--    &0.89 &0.95 &1.01  \\ \hline
$ L^{\infty}$-norm &  3.05E-2  &   1.63E-2  &   8.41E-3 & 4.19E-3  \\
    Order &--    &0.90 &0.96 &1.00  \\ \hline
\text{Manifold distance}   & 6.89E-2 &  3.65E-2 & 1.86E-2 &9.17E-3   \\
    Order  &--    &0.92 &0.97 &1.02 \\ \hline
\text{Hausdorff distance}   & 3.04E-2 &  1.62E-2 & 8.29E-3 &4.09E-3 \\
    Order  &--    &0.91 &0.97 &1.02
 \end{tabular*}
{\rule{\temptablewidth}{1pt}}
\label{Tab:Different_norm_ellipse_Dziuk}
\end{table}

\begin{table}[h!]
\def\temptablewidth{0.95\textwidth}
\caption{Numerical errors quantified  by various metrics for the BGN1 scheme, with the parameters $N=10000,\tau_0=1/40$, $T=0.25$.}
{\rule{\temptablewidth}{1pt}}
\begin{tabular*}{\temptablewidth}{@{\extracolsep{\fill}}llllll}
\text{Errors}
 & $\tau=\tau_0$ & $\tau_0/2$ &$\tau_0/2^2$  &$\tau_0/2^3$ \\\hline
$ L^2$-norm  & 4.25E-3 &3.98E-3 &4.05E-3 &4.15E-3  \\
    Order &--    &0.10 &$-$0.03 &$-$0.03  \\ \hline
$ L^{\infty}$-norm  &   1.00E-2  &   9.17E-3 & 9.47E-3 &9.79E-3  \\
    Order &--    &0.12 &$-$0.05 &$-$0.05  \\ \hline
\text{Manifold distance}  & 3.11E-2 &  1.58E-2 & 7.96E-3 &4.00E-3   \\
    Order  &--    &0.98 &0.99 &0.99 \\ \hline
\text{Hausdorff distance}   & 8.23E-3 &  4.18E-3 & 2.11E-3 &1.06E-3 \\
    Order  &--    &0.98 &0.99 &0.99
 \end{tabular*}
{\rule{\temptablewidth}{1pt}}
\label{Tab:Different_norm_ellipse_BGN1}
\end{table}

\begin{table}[h!]
\def\temptablewidth{0.95\textwidth}
\caption{Numerical errors quantified  by various metrics for the BGN2 scheme, with the parameters $N=10000,\tau_0=1/40$, $T=0.25$.}
{\rule{\temptablewidth}{1pt}}
\begin{tabular*}{\temptablewidth}{@{\extracolsep{\fill}}llllll}
\text{Errors}
 & $\tau=\tau_0$ & $\tau_0/2$ &$\tau_0/2^2$  &$\tau_0/2^3$  \\\hline
$ L^2$-norm  & 1.49E-2 &1.45E-2 &1.45E-2 &1.43E-2  \\
    Order &--    &0.04 &0.00 &0.02 \\ \hline
$ L^{\infty}$-norm   &   3.32E-2  &   3.30E-2 & 3.29E-2 &3.29E-2  \\
    Order &--     &0.01 &0.00 &0.00  \\ \hline
\text{Manifold distance}  & 8.44E-4 &  2.11E-4 & 5.27E-5 &1.32E-5   \\
    Order  &--    &2.00 &2.00 &1.99 \\ \hline
\text{Hausdorff distance}  & 2.00E-4 &  4.98E-5 & 1.26E-5 &3.29E-6 \\
    Order  &--    &2.01 &1.98 &1.94
 \end{tabular*}
{\rule{\temptablewidth}{1pt}}
\label{Tab:Different_norm_ellipse_BGN2}
\end{table}

It is worth noting that unlike Dziuk's scheme, the convergence of the BGN1 scheme and BGN2 scheme under function metrics (the $L^2$-norm and $L^\infty$-norm) is not as satisfactory. This is not surprising since the error in classical Sobolev space depends on the specific parametrization of the curve. In contrast, the BGN formulation \eqref{CSF:Coupled equation} allows tangential motion to make the mesh points equidistribute, which indeed affects the parametrization while preserving the shape of the curve. Thus it is not appropriate to use the classical function metrics to quantify the errors of the BGN-type schemes which are based on the BGN formulation.
Instead, as observed from Tables \ref{Tab:Different_norm_ellipse_BGN1} and \ref{Tab:Different_norm_ellipse_BGN2}, the shape metrics are much more suitable for quantifying the numerical errors of the schemes that allow intrinsic tangential velocity. In the remaining of the article, we will employ the manifold distance or the Hausdorff distance when measuring the difference between two curves.

\section{Applications to other geometric flows}

In this section, we extend the above proposed BGN2 scheme to other geometric flows.

\subsection{For area-preserving curve-shortening flow (AP-CSF)}
As is known, the AP-CSF can be viewed as the $L^2$-gradient flow with respect to the length functional under the constraint of total area preservation~\cite{BGN20,Jiang23}. Similar to \eqref{CSF:Coupled equation}, we rewrite the AP-CSF as the following coupled equations
\begin{equation}\label{AP-CSF:Coupled equation}
\begin{split}
	\p_t \mathbf{X}\cdot \mathbf{n} &=-\kappa+\l<\kappa \r> ,\\
			\kappa \mathbf{n}&=-\p_{ss}\mathbf{X},
\end{split}
\end{equation}
where the average of curvature is defined as  $\l<\kappa\r>:=\int_{\Gamma(t)}\kappa \d s /\int_{\Gamma(t)}1 \d s$.
The fully-discrete, first-order in time semi-implicit BGN scheme for AP-CSF reads as~\cite{BGN20}:

(\textbf{BGN1 scheme for AP-CSF}): For $m\ge 0$, find $\mathbf{X}^{m+1}\in [V^h]^2$ and $\kappa^{m+1}\in V^h$ such that
\begin{equation}\label{AP-CSF:BGN1}
	\begin{cases}
			\l(\frac{\mathbf{X}^{m+1}-\mathbf{X}^m}{\tau},\varphi^h \mathbf{n}^m \r)^h_{\Gamma^m}+\l(  \kappa^{m+1}-\l< \kappa^{m+1}\r>_{\Gamma^m} ,\varphi^h \r)_{\Gamma^m}^h=0,\\
	\l(\kappa^{m+1},\mathbf{n}^m\cdot \bm{\omega}^h\r)_{\Gamma^m}^h-\l(\p_s \mathbf{X}^{m+1},\p_s\bm{\omega}^h\r)_{\Gamma^m}=0,
		\end{cases}
\end{equation}
for all $(\varphi^h, \bm{\omega}^h)\in V^h\times [V^h]^2$,
where $\l<\kappa^{m+1}\r>_{\Gamma^m}:=\l(\kappa^{m+1},1 \r)_{\Gamma^m}^h/\l(1,1 \r)_{\Gamma^m}^h$.

Based on the same spirit, we can propose the following second-order BGN2 scheme.

(\textbf{BGN2 scheme for AP-CSF}):~ For $m\ge 1$, find $(\mathbf{X}^{m+1}, \kappa^{m+1})\in [V^h]^2\times V^h$ such that
\begin{equation}
\label{AP-CSF:BGN2}
\begin{cases}
\big(\frac{\mathbf{X}^{m+1}-\mathbf{X}^{m-1}}{2\tau},\varphi^h \mathbf{n}^m \big)^h_{\Gamma^m}=-\big(  \frac{\kappa^{m+1}+\kappa^{m-1}}{2}-\big<\frac{\kappa^{m+1}+\kappa^{m-1}}{2} \big>_{\Gamma^m},\varphi^h \big)_{\Gamma^m}^h,\\
\big(\frac{\kappa^{m+1}+\kappa^{m-1}}{2},\mathbf{n}^m\cdot \bm{\omega}^h\big)_{\Gamma^m}^h-\big(\frac{\p_s \mathbf{X}^{m+1}+\p_s \mathbf{X}^{m-1}}{2},\p_s\bm{\omega}^h\big)_{\Gamma^m}=0,
\end{cases}
\end{equation}
for all $(\varphi^h, \bm{\omega}^h)\in V^h\times [V^h]^2$.
Similarly, the stiffness matrix of the linear system to be solved in \eqref{AP-CSF:BGN2} is exactly the same as the BGN1 scheme \eqref{AP-CSF:BGN1}, whose well-posedness has been established in \cite[Theorem 90]{BGN20}. 
The equi-BGN1 scheme \cite{BGN20} and equi-BGN2 scheme can be derived in a similar manner. Similarly, unconditional interlaced energy stability for  the equi-BGN2 scheme can be obtained.





\subsection{For surface diffusion flow (SDF)}
We consider the fourth-order flow---SDF, which can be viewed as the $H^{-1}$-gradient flow with respect to the length functional~\cite{BGN20,Bao-Zhao}. In a similar fashion, we rephrase the SDF as the subsequent system of equations
\begin{equation}\label{SD:Coupled equation}
\begin{split}
	\p_t \mathbf{X}\cdot \mathbf{n} &=\p_{ss} \kappa ,\\
			\kappa \mathbf{n}&=-\p_{ss}\mathbf{X}.
\end{split}
\end{equation}
The fully discrete, first-order in time semi-implicit BGN scheme for SDF reads as~\cite{BGN07A}:

(\textbf{BGN1 scheme for SDF}):~For $m\ge 0$, find $\mathbf{X}^{m+1}\in [V^h]^2$ and $\kappa^{m+1}\in V^h$ such that
\begin{equation}\label{SD:BGN1}
	\begin{cases}
			\l(\frac{\mathbf{X}^{m+1}-\mathbf{X}^m}{\tau},\varphi^h \mathbf{n}^m \r)^h_{\Gamma^m}+\l( \p_s \kappa^{m+1}, \p_s\varphi^h \r)_{\Gamma^m}=0,\quad \forall\ \varphi^h\in V^h,\\
	\vspace{-3mm}		\l(\kappa^{m+1},\mathbf{n}^m\cdot \bm{\omega}^h\r)_{\Gamma^m}^h-\l(\p_s \mathbf{X}^{m+1},\p_s\bm{\omega}^h\r)_{\Gamma^m}=0,\quad \forall\ \bm{\omega}^h\in  [V^h]^2.
		\end{cases}
\end{equation}
In line with the same approach, we can put forward the subsequent second-order BGN2 scheme:

(\textbf{BGN2 scheme for SDF}):~ For $m\ge 1$,  find $(\mathbf{X}^{m+1}, \kappa^{m+1})\in [V^h]^2\times V^h$ such that
\begin{equation}\label{SD:BGN2}
\begin{cases}
\l(\frac{\mathbf{X}^{m+1}-\mathbf{X}^{m-1}}{2\tau},\varphi^h \mathbf{n}^m \r)^h_{\Gamma^m}+\l(  \frac{\p_s\kappa^{m+1}+\p_s\kappa^{m-1}}{2},\p_s\varphi^h \r)_{\Gamma^m}=0,\\
\vspace{-5mm}\\
\l(\frac{\kappa^{m+1}+\kappa^{m-1}}{2},\mathbf{n}^m\cdot \bm{\omega}^h\r)_{\Gamma^m}^h-\l(\frac{\p_s \mathbf{X}^{m+1}+\p_s \mathbf{X}^{m-1}}{2},\p_s\bm{\omega}^h\r)_{\Gamma^m}=0,
\end{cases}
\end{equation}
for all $(\varphi^h, \bm{\omega}^h)\in V^h\times [V^h]^2$.
The well-posedness 
of the above scheme can be shown similarly under certain mild conditions.

For the schemes \eqref{AP-CSF:BGN2} and \eqref{SD:BGN2}, we consistently set $\bX^0\in [V^h]^2$ as specified in Algorithm \ref{CSF:BGN initial data 1}, that is, $\bX^0$ is a parametrization of an (almost) equidistributed interpolation polygon with $N$ vertices for the initial curve $\Gamma(0)$. Similar as the case of CSF, to start the BGN2 schemes, we need to prepare the initial data $\kappa^0$ and $(\bX^1,\kappa^1)$, which can be achieved by using the similar approach as Algorithm \ref{CSF:BGN initial data 1} by using the corresponding BGN1 scheme. A complete second-order scheme can be obtained as in Algorithm \ref{Full algorithm} with the corresponding BGN1 scheme as a mesh regularization  when necessary.

\section{Numerical results}
\label{sec:illust}

\subsection{Convergence tests}
\label{sec:order 2, illu}

In this subsection, we test the temporal convergence of the second-order schemes \eqref{CSF:BGN2}, \eqref{AP-CSF:BGN2} and \eqref{SD:BGN2} for solving the three geometric flows: CSF, AP-CSF and SDF, respectively, with two different mesh regularization techniques. As previously discussed in
Section \ref{sec:Different error}, we quantify the numerical errors of the curves  using the shape metrics, such as the manifold distance
and Hausdorff distance. For the following simulations, we select four distinct types of  initial shapes:
\smallskip

\begin{itemize}
\item (\textbf{Shape 1}): a unit circle;
\item (\textbf{Shape 2}): an ellipse with semi-major axis $2$ and semi-minor axis $1$;
\item (\textbf{Shape 3}): a `tube' shape, which is a curve comprising a $4 \times 1$ rectangle with two semicircles on its left and right sides;
\item (\textbf{Shape 4}): a `flower' shape, which is parameterized by
\begin{equation*}
\bX(\rho)=((2+\cos(12\pi\rho))\cos(2\pi\rho),(2+\cos(12\pi\rho))\sin(2\pi\rho)),\quad \rho\in \mathbb{I}=[0,1].
\end{equation*}
\end{itemize}
We note that for the CSF with Shape 1 as its initial shape has the following true solution, i.e.,
\[\bX_{\mathrm{true}}(\rho,t)=\sqrt{1-2t}(\cos(2\pi\rho),\sin(2\pi\rho)),\quad \rho\in \mathbb{I},\quad t\in [0,0.5).\]
For this particular case, we compute the numerical error by comparing it with the true solution. However, for all other cases, we utilize the reference solutions which are obtained by the BGN2 scheme with large $N$ and a tiny time step size $\tau$.  In addition, the mesh regularization threshold for the BGN2 scheme is consistently set to $n_{\text{MR}}=10$, and the iteration tolerance of the equi-BGN2 scheme is set as $\mathrm{tol}=10^{-10}$.


\begin{figure}[h!]
\vspace{1mm}
\hspace{1mm}
\includegraphics[width=4.6in,height=2in]{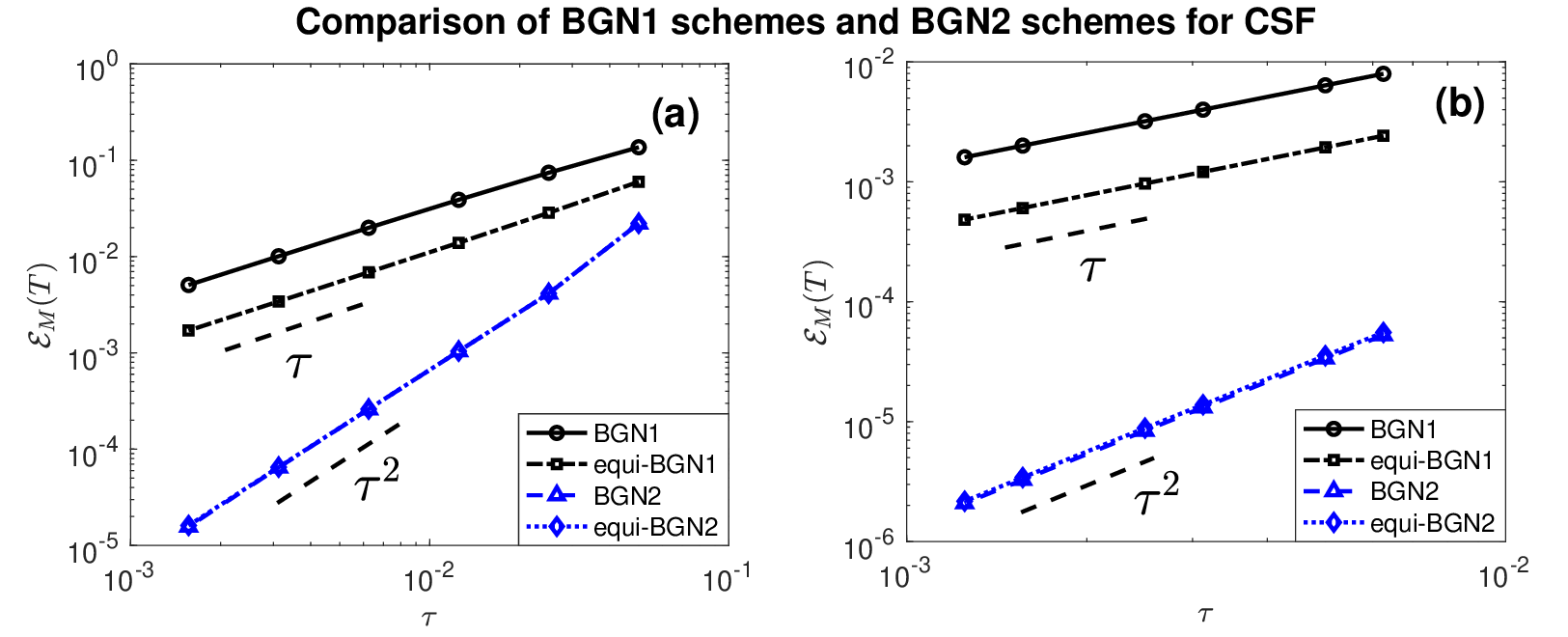}
\vspace{-4mm}
\caption{Log-log plot of the numerical errors at time $T=0.25$ measured by the manifold distance for BGN1, equi-BGN1, BGN2 and equi-BGN2 schemes for solving the CSF with  two different initial curves: (a) Shape 1 and (b) Shape 2, respectively, where the number of nodes is fixed as $N=10000$.}
\label{Fig:CSF_EOC1}
\end{figure}

\begin{figure}[h!]
\hspace{1mm}
\includegraphics[width=4.6in,height=2in]{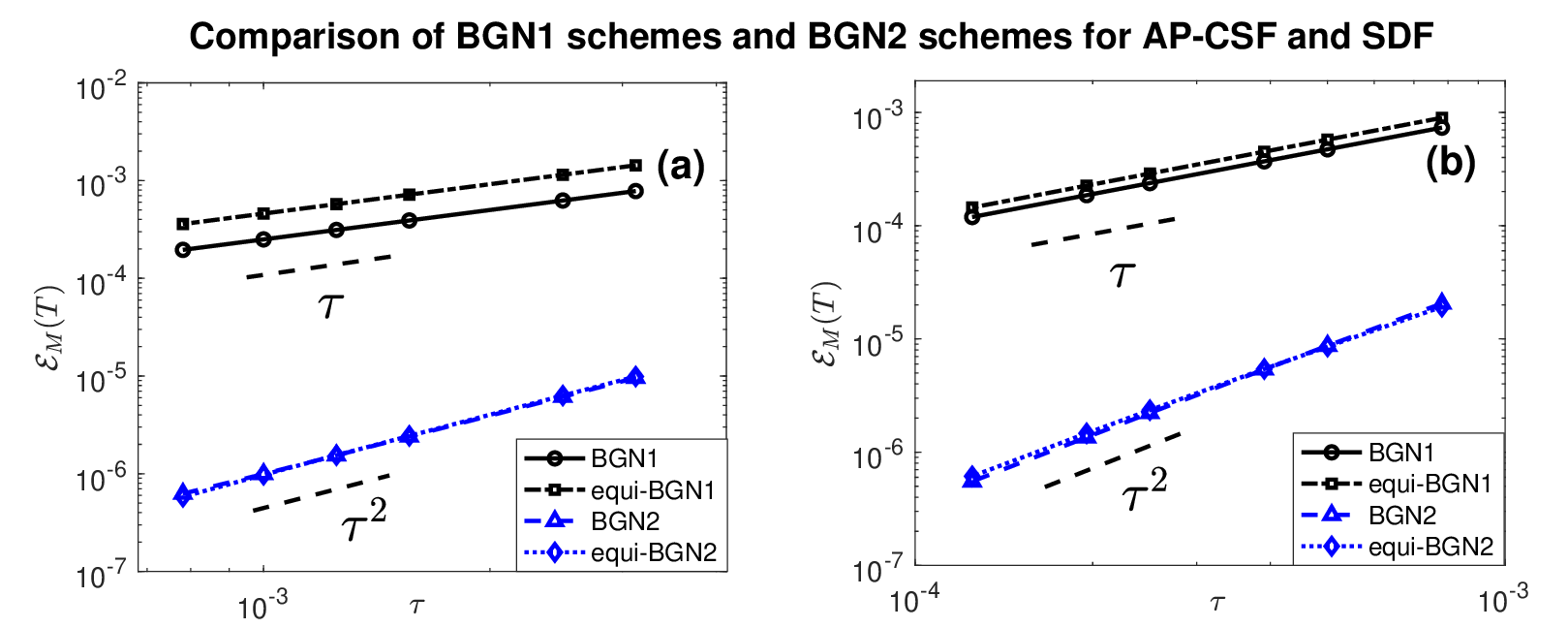}
\vspace{-4mm}
\caption{Log-log plot of the numerical errors at time $T=0.25$, measured by the manifold distance, for solving two different flows with Shape 2 as the initial curve: (a) AP-CSF and (b) SDF, respectively.}
\label{Fig:APCSFandSDF_EOC1}
\end{figure}

We begin our test by calculating the convergence of the BGN2 scheme and the equi-BGN2 scheme  for the CSF with either Shape 1 or Shape 2 as initial data. Fig.~\ref{Fig:CSF_EOC1} presents a log-log plot of the numerical errors at time $T=0.25$, measured by the manifold distance. The errors for the Hausdorff distance, which are similar, are not included here for brevity. To ensure a fair comparison, we also include the numerical results of the BGN1 scheme \eqref{CSF:BGN1} and the equi-BGN1 scheme \eqref{CSF:equi-BGN1} under the same computational parameters, with a fixed number of grid points $N=10000$. As clearly shown in Fig.~\ref{Fig:CSF_EOC1}, the numerical error of the BGN2-type schemes  reduce very rapidly with second-order accuracy in time, while the BGN1-type schemes only achieve first-order convergence.



Fig.~\ref{Fig:APCSFandSDF_EOC1} illustrates the temporal errors of the schemes for solving the AP-CSF and SDF with Shape 2 as initial data, showing quadratic convergence for BGN2-type schemes and linear convergence for BGN1-type schemes.


\subsection{Comparison of computational costs}
\label{sec:cost}

In order to show that the computational cost of the proposed BGN2 schemes is comparable to that of the BGN1 schemes, we present two examples about solving the CSF and SDF, respectively. The numerical codes were written by using  MATLAB 2021b, and they were implemented in a MacBook Pro with 1.4GHz quad-core Intel Core i5 and 8GB RAM.

Table~\ref{tab:CPU time} displays a comparison of CPU times in seconds and numerical errors at time $T=0.05$, as measured by the manifold distance $\cE_{M}(T)$ and Hausdorff distance $\cE_{H}(T)$, using the BGN1-type and BGN2-type schemes for solving the CSF, where the initial shape is chosen as Shape 1. Table~\ref{tab:CPU2} provides similar results for solving the SDF with Shape 3 as its initial shape. Based on the findings presented in Tables \ref{tab:CPU time} and \ref{tab:CPU2}, the following conclusions can be drawn.  (i) On the same mesh, the computational cost of the BGN2 scheme is slightly higher than that of the BGN1 scheme, as it involves additional calculations for the initial values and the right-hand side of the linear system at each time level. Meanwhile, the equi-BGN2 scheme incurs more or less similar computational cost as the equi-BGN1 scheme.
However, the numerical solutions obtained using the BGN2-type schemes are significantly more accurate than those of the BGN1-type schemes; (ii) The computational cost of the equi-BGN2 scheme is several times higher than that of the BGN2 scheme, since it needs to solve a nonlinear system at each time step. However, equidistribution and unconditional energy stability can be theoretically guaranteed for the equi-BGN2 scheme.
\begin{table}[h]
\centering
\caption{Comparisons of the CPU times (seconds) and the numerical errors measured from the manifold distance $\cE_{M}(T)$ and Hausdorff distance $\cE_{H}(T)$ for the BGN1-type and BGN2-type schemes applied to CSF},  where the initial shape is chosen as Shape 1, with $\tau=0.5/N$ and $T=0.05$.
\vspace{0.5cm}
\begin{tabular}{|c|c|c|c|c|c|c|c|}
\hline
\multicolumn{4}{|c|}{BGN1 scheme  } & \multicolumn{4}{c|}{BGN2 scheme }\\ \hline
	   $N$  & $\cE_{M}(T)$  & $\cE_{H}(T)$ &    {\text{Time}(s)}   & $N$ &  $\cE_{M}(T)$ & $\cE_{H}(T)$  &   {\text{Time}(s)}        \\ \hline
	 320  &  5.61E-4  &  1.25E-4  &  0.350   & 320  & 2.09E-4  & 5.04E-5  & 0.430     \\ \hline
       640   & 3.34E-4  & 6.37E-5 & 1.70 & 640  & 5.20E-5  & 1.27E-5  & 2.30      \\ \hline
	1280    & 1.81E-4  & 3.22E-5  & 9.85 & 1280  & 1.29E-5  & 3.20E-6  & 12.9      \\ \hline
      2560   & 9.38E-5  & 1.62E-5 & 110 & 2560  & 3.08E-6  & 8.16E-7  & 130
          \\ \hline\multicolumn{4}{|c|}{equi-BGN1 scheme } & \multicolumn{4}{c|}{equi-BGN2 scheme }\\ \hline
	   $N$  & $\cE_{M}(T)$  & $\cE_{H}(T)$ &    {\text{Time}(s)}   & $N$ &  $\cE_{M}(T)$ & $\cE_{H}(T)$  &   {\text{Time}(s)}        \\ \hline
	 320  &  4.70E-4   &  9.42E-5   &  1.16    & 320  & 2.09E-4  & 5.03E-5  & 0.82   \\ \hline
       640   &  1.82E-4  & 3.44E-5  & 4.93  & 640  & 5.20E-5  & 1.26E-5  & 4.19      \\ \hline
	1280    & 7.78E-5  & 1.40E-5  & 25.5  &  1280  & 1.29E-5  & 3.14E-6  & 25.4      \\ \hline
      2560   & 3.55E-5  & 6.22E-6  &  284 &
2560  & 3.08E-6  &  7.86E-7 & 304       \\ \hline
\end{tabular}
\label{tab:CPU time}
\end{table}

\begin{table}[h]
\centering
\caption{Comparisons of the CPU times (seconds) and the numerical errors measured by the manifold distance $\cE_{M}(T)$ and Hausdorff distance $\cE_{H}(T)$ using the BGN1-type and BGN2-type schemes applied to SDF}, where the initial shape is chosen as Shape 3, with $\tau=0.5/N$, and $T=0.05$.
\vspace{0.5cm}
\begin{tabular}{|c|c|c|c|c|c|c|c|}
\hline
\multicolumn{4}{|c|}{BGN1 scheme  } & \multicolumn{4}{c|}{BGN2 scheme }\\ \hline
	   $N$  & $\cE_{M}(T)$  & $\cE_{H}(T)$ &    {\text{Time}(s)}   & $N$ &  $\cE_{M}(T)$ & $\cE_{H}(T)$  &   {\text{Time}(s)}       \\ \hline
      320   & 4.73E-3  & 6.91E-4 & 0.470 &320  & 2.53E-3  & 1.14E-3  & 0.610      \\ \hline
	 640  & 2.24E-3  & 3.38E-4  & 2.03 & 640  & 8.28E-4  & 4.17E-4  & 2.27      \\ \hline
     1280   & 1.10E-3  & 1.67E-4 & 12.6 & 1280  & 2.30E-4  & 1.12E-4  & 15.1      \\ \hline
      2560   & 5.53E-4  & 8.34E-5 & 133 & 2560  & 5.42E-5  & 2.82E-5  & 153      \\\hline
       \multicolumn{4}{|c|}{equi-BGN1 scheme  } & \multicolumn{4}{c|}{equi-BGN2 scheme }\\ \hline
	   $N$  & $\cE_{M}(T)$  & $\cE_{H}(T)$ &    {\text{Time}(s)}   & $N$ &  $\cE_{M}(T)$ & $\cE_{H}(T)$  &   {\text{Time}(s)}        \\ \hline
	 320  &   5.00E-3  &  1.04E-3   &  3.39    & 320  & 2.71E-3  & 1.21E-3  & 3.48  \\ \hline
       640   & 2.62E-3   & 5.61E-4  & 17.1  & 640  &  8.88E-4 & 4.33E-4  & 16.7      \\ \hline
	1280    &  1.34E-3 & 2.93E-4  & 105  &  1280  & 2.64E-4  & 1.56E-4 &  102     \\ \hline
      2560   & 6.83E-4  & 1.51E-4 & 1151 &  2560  & 8.12E-5  & 5.57E-5  & 1140      \\ \hline
\end{tabular}
\label{tab:CPU2}
\end{table}

\subsection{Applications to the curve evolution}
\label{sec:long time, illu}

As is well-known, the AP-CSF and SDF possess some structure-preserving properties, such as the perimeter decreasing and area conserving properties~\cite{JiangLi21, Jiang23, Bao-Zhao}.
In this subsection, we investigate the structure-preserving properties of the proposed BGN2 scheme and equi-BGN2 scheme applied to AP-CSF and SDF, respectively.  As an example, we mainly focus on the SDF here. Moreover, we will discuss the importance of the mesh regularization procedures.

 Fig.~\ref{Fig:Ellipse_evo_and_Geo} (a) illustrates the evolution of an initially elliptic curve, referred to as Shape 2, driven by SDF towards its equilibrium state by the BGN2 scheme. Fig.~\ref{Fig:Ellipse_evo_and_Geo}(b)-(e) show the evolution of various geometric quantities during the process: the relative area loss $\Delta A(t)$, the normalized perimeter $L(t)/L(0)$, and the mesh distribution function $\Psi(t)$,
 which are defined respectively as
\[
\Delta A(t)|_{t=t_m}=\frac{A^m-A^0}{A^0},\quad \l.\frac{L(t)}{L(0)}\r|_{t=t_m}=\frac{L^m}{L^0},\quad \Psi(t)|_{t=t_m}=\frac{\max_j|\mathbf{h}^m_j| }{\min_j|\mathbf{h}^m_j|},
\]
for $m\ge 0$, where $A^m$ is the area enclosed by the polygon determined by $\bX^m$, $L^m$ represents the perimeter of the polygon, and $\Psi(t)$ is the mesh ratio function. As depicted in Fig.~\ref{Fig:Ellipse_evo_and_Geo}(b), the area loss exhibits a weakly oscillating behavior, which may result from the two-step structure of the BGN2 scheme, the equi-BGN2 scheme has similar oscillating behavior and we omit it here for brevity. It is worth noting that despite the oscillations, the normalized area loss remains very low, consistently  below $0.01\%$. By employing a smaller grid size, the area loss can be further reduced, and it is significantly lower than that of the BGN1 scheme under the same computational parameters. Furthermore, Fig.~\ref{Fig:Ellipse_evo_and_Geo}(c) shows the BGN2 scheme and the equi-BGN2 scheme preserve the perimeter-decreasing property of the SDF numerically. Furthermore, in Fig.~\ref{Fig:Ellipse_evo_and_Geo}(d), it can be observed that the mesh distribution function $\Psi(t)$ remains lower than $1.2$ during the evolution. This indicates that the mesh distribution remains well-maintained and almost equidistributed during the process. Therefore, in this scenario, there is no need to perform the mesh regularization procedure because $\Psi(t)$ is always smaller than the chosen threshold $n_{\text{MR}}$ (here we choose it as $10$) in the simulations. Additionally, Fig.~\ref{Fig:Ellipse_evo_and_Geo}(e) shows the equi-BGN2 scheme achieves  equidistribution property at each time step.  The relatively low iteration numbers do not compromise the computational efficiency.

\begin{figure}[h!]
\hspace{-2mm}
\includegraphics[width=4.8in,height=3.3in]{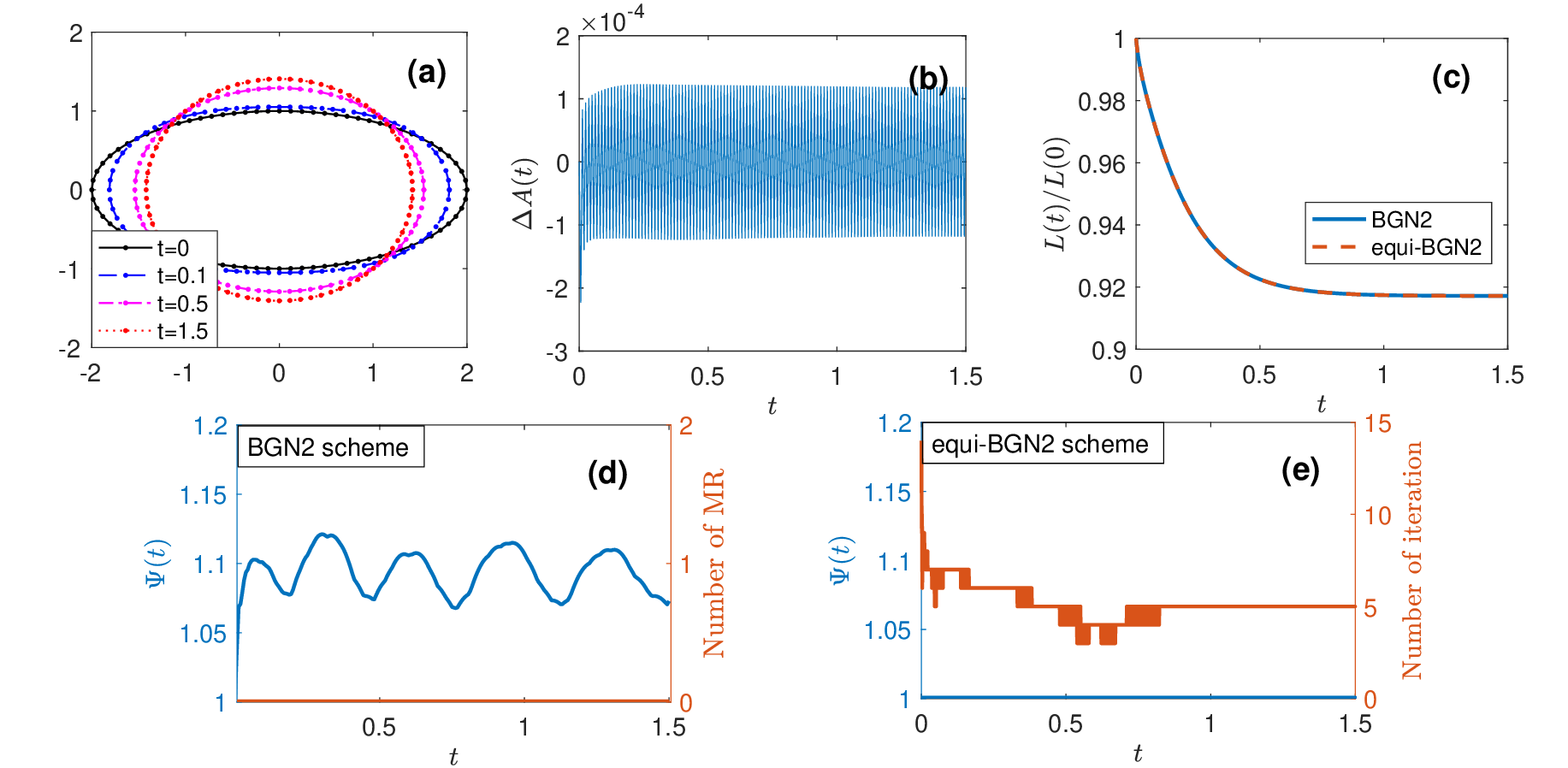}
\vspace{-9mm}
\caption{(a) Several snapshots of the curve evolution controlled by the SDF, starting with Shape 2 as its initial shape. (b) The relative area loss as a function of time. (c) The normalized perimeter as a function of time. (d) The mesh ratio function $\Psi(t)$ (in blue line) and the number of mesh regularizations (in red line) for the BGN2 scheme. (e) The mesh ratio function $\Psi(t)$ (in blue line) and the number of iteration numbers (in red line) at each time step for the equi-BGN2 scheme}. For (a)-(b), we used $N=80$ and $\tau=1/160$ while for (c)-(e), $N=640$ and $\tau=1/1280$.
\label{Fig:Ellipse_evo_and_Geo}
\end{figure}

To provide a more comprehensive comparison, we conduct simulations of evolution of Shape 3 curve driven by the SDF. Fig.~\ref{Fig:tube_evo_and_Geo}(b)-(c) demonstrates that the BGN2 scheme and the equi-BGN2 scheme effectively preserve two crucial geometric properties
of the SDF: the conservation of area and the reduction of perimeter properties~\cite{JiangLi21, Bao-Zhao}. It should be noted that
Fig.~\ref{Fig:tube_evo_and_Geo}(d) reveals that without the implementation of mesh regularization, the mesh distribution function $\Psi(t)$ can become very large.
Therefore, in our algorithm, when $\Psi(t)$ exceeds a threshold $n_{\text{MR}}$,
we employ the BGN1 scheme \eqref{SD:BGN1} for a single run to perform mesh regularization, similar to $\bm{Step~3}$  of Algorithm \ref{Full algorithm}. As clearly shown in Fig.~\ref{Fig:tube_evo_and_Geo}(d), following this step, the mesh ratio rapidly decreases to a low value, which makes the method more stable. Importantly, this mesh regularization procedure is only required four times throughout the entire evolution, without sacrificing the accuracy of the BGN2 scheme (cf. Table  \ref{tab:CPU2}). Similarly, as shown in Fig.~\ref{Fig:tube_evo_and_Geo}(e), the equi-BGN2 scheme also performs well for this initial shape. Compared to the case of Shape 2,  although we require more iteration steps, it is still superior to the BGN1 scheme in view of the accuracy and efficiency (cf. Table  \ref{tab:CPU2}).

\begin{figure}[h!]
\hspace{-2mm}
\includegraphics[width=4.8in,height=3.3in]{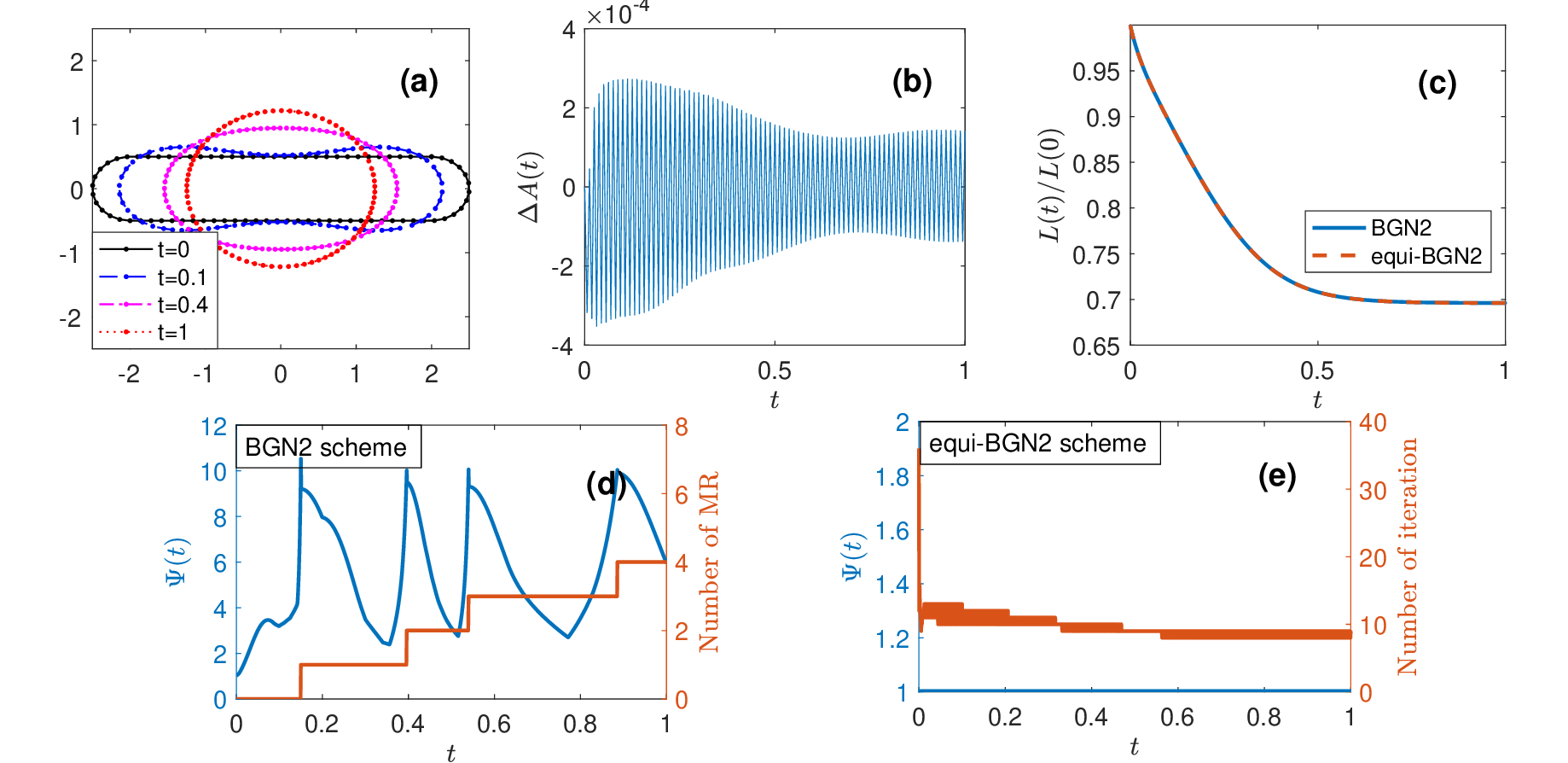}
\vspace{-9mm}
\caption{(a) Several snapshots of the curve evolution controlled by the
 SDF, starting with Shape 3 as its initial shape. (b) The relative area loss as a function of time. (c) The normalized perimeter as a function of time. (d) The mesh distribution function $\Psi(t)$ (in blue line) and the number of mesh regularizations (in red line) for the BGN2 scheme. (e) The mesh ratio function $\Psi(t)$ (in blue line) and the number of iteration numbers (in red line) at each time step for the equi-BGN2 scheme}. For (a)-(b) we used $N=80$ and $\tau=1/160$ while $N=640$ and $\tau=1/1280$ for (c)-(e).
\label{Fig:tube_evo_and_Geo}
\end{figure}

Next, we proceed to simulate the evolution of a nonconvex curve, referred to as Shape 4. Fig.~\ref{Fig:flower_geo_MRandnoMR} and Fig.~\ref{Fig:flower_geo_MRandnoMRb} (the first row) show the evolution of the geometric quantities based on two different initial data preparations: Algorithm \ref{CSF:BGN initial data 1} and Remark \ref{iniitalp}, respectively. A comparison of the results reveals the superiority of the latter approach for several reasons: (i) the magnitude of area loss is significantly lower when using the approach in Remark \ref{iniitalp}; (ii) the perimeter-decreasing property is preserved while the perimeter oscillates at the beginning when using Algorithm \ref{CSF:BGN initial data 1}; (iii) the number of mesh regularization implementations is smaller with the approach in Remark \ref{iniitalp}. Thus we recommend preparing the data for a nonconvex initial curve following the approach outlined in Remark \ref{iniitalp}. Fig. \ref{Fig:flower_geo_MRandnoMRb} (the second row) demonstrates the performance  of the equi-BGN2 scheme, from which it can be seen that only a relatively low number of iterations are needed for the majority of time steps (see \ref{Fig:flower_geo_MRandnoMRb}(c2)). Additionally, Fig. \ref{Fig:flower_geo_MRandnoMRb} (the third row) illustrates the evolution of the same quantities without any implementations of mesh regularization. In this case, all three quantities exhibit significant oscillations  after a certain time period, and the area loss and mesh ratio of the polygon becomes excessively large, resulting in the breakdown of the BGN2 scheme. Notably, mesh clustering has happened at $t=1$ (see Fig. \ref{Fig:flower_evo_MRandnoMR}(c3)), eventually  leading to mesh distortion at $t=2$ (see Fig. \ref{Fig:flower_evo_MRandnoMR}(d3)).

\begin{figure}[h]
\hspace{-10mm}
\includegraphics[width=5.3in,height=1.0in]{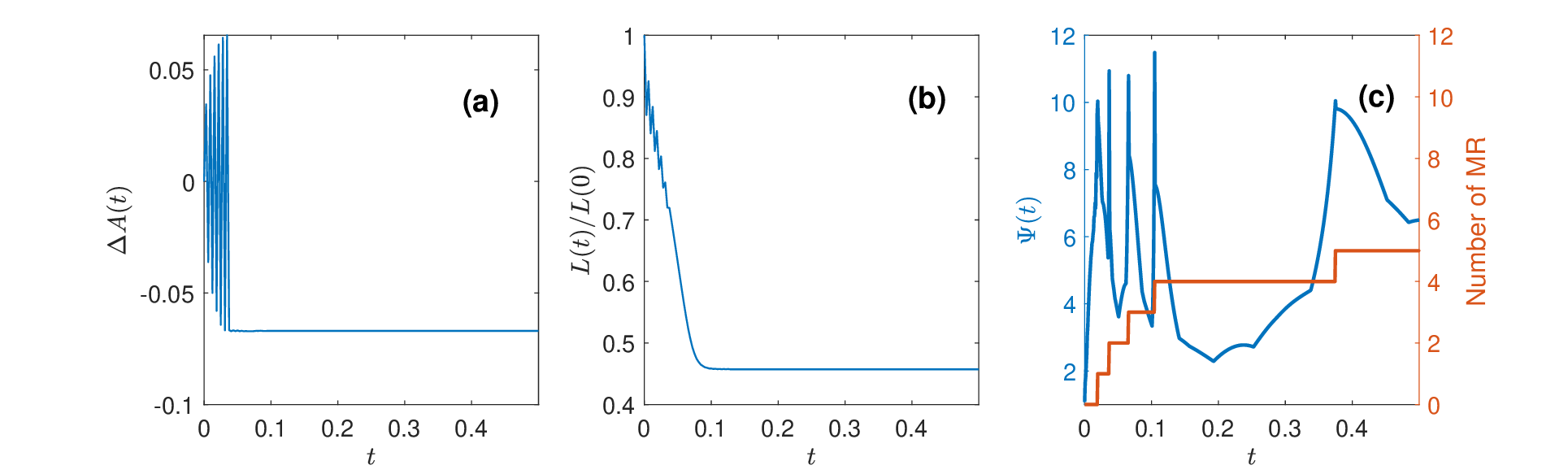}
\vspace{-10mm}
\caption{Evolution of the three geometrical quantities when the initial data is prepared as in Algorithm \ref{CSF:BGN initial data 1}: (a) the relative area loss, (b) the normalized perimeter, (c) the mesh distribution function $\Psi(t)$, for  the BGN2 scheme.}
\label{Fig:flower_geo_MRandnoMR}
\end{figure}

\vspace{-5mm}

\begin{figure}[h]
\hspace{-1mm}
\includegraphics[width=4.6in,height=3.0in]{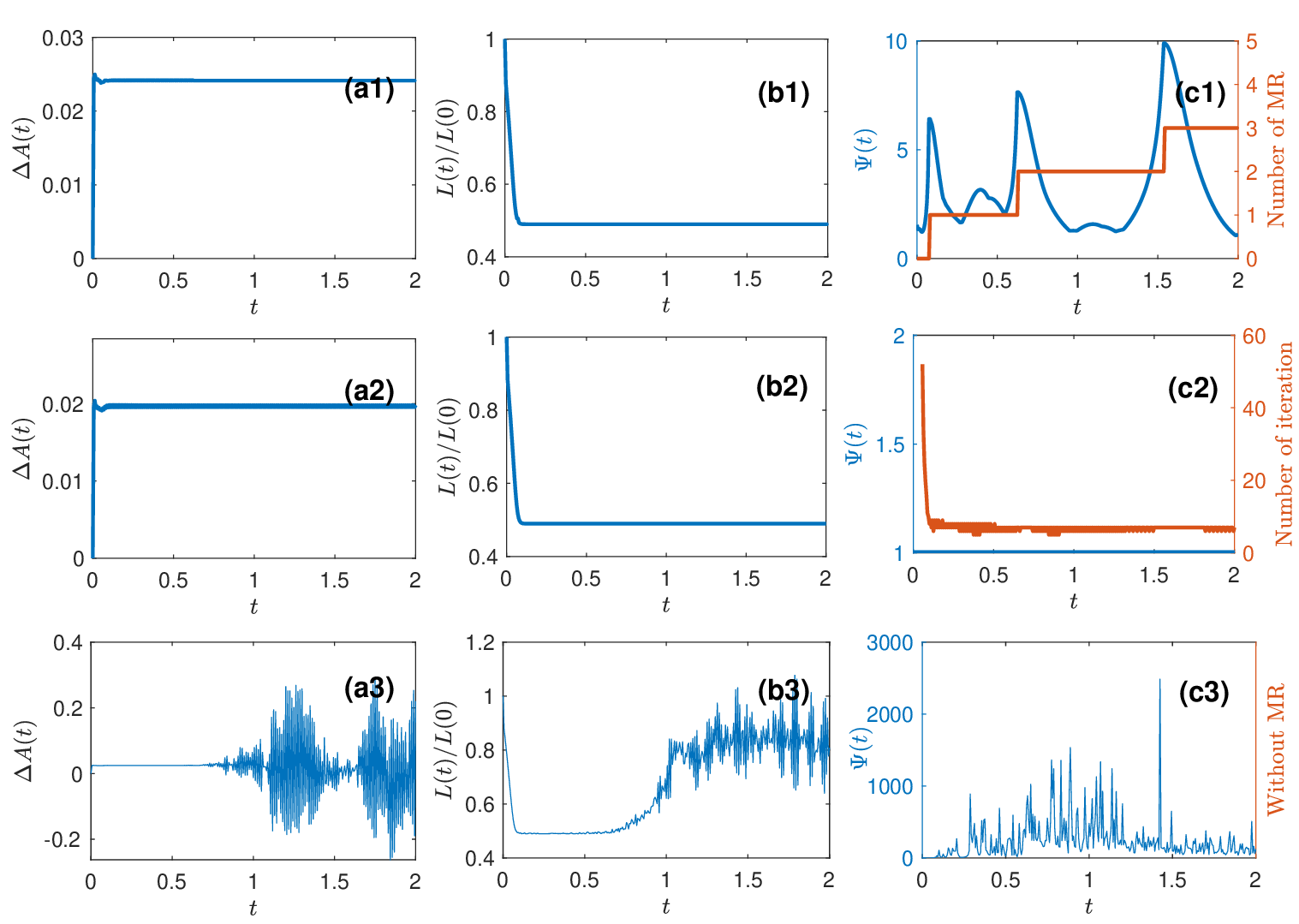}
\vspace{-3mm}
\caption{Evolution of the three geometrical quantities when the initial data is prepared as in Remark \ref{iniitalp}: (a) the relative  area loss, (b) the normalized perimeter, (c) the mesh distribution function $\Psi(t)$,  for the BGN2 scheme (shown in the first row), the equi-BGN2 scheme (shown in the second row) and without mesh regularization procedure (shown in the third row)}.
\label{Fig:flower_geo_MRandnoMRb}
\end{figure}

These issues can be avoided by implementing one of the mesh regularization techniques. Fig. \ref{Fig:flower_evo_MRandnoMR}(a1)-(d1) and  Fig. \ref{Fig:flower_evo_MRandnoMR}(a2)-(d2)  demonstrate that mesh regularization is crucial for the effectiveness of BGN2-type schemes and the BGN1-type schemes can significantly enhance mesh quality. Additionally, A comparison between Fig. \ref{Fig:flower_evo_MRandnoMR}(d1) and Fig. \ref{Fig:flower_evo_MRandnoMR}(d2) reveals that there still exists some mesh clustering for the BGN2 scheme in long-time evolution.  In contrast, the equi-BGN2 scheme exhibits equidistribution property throughout all time.

\begin{figure}[h]
\hspace{-3mm}
\centering
\includegraphics[width=4.5in,height=2.6in]{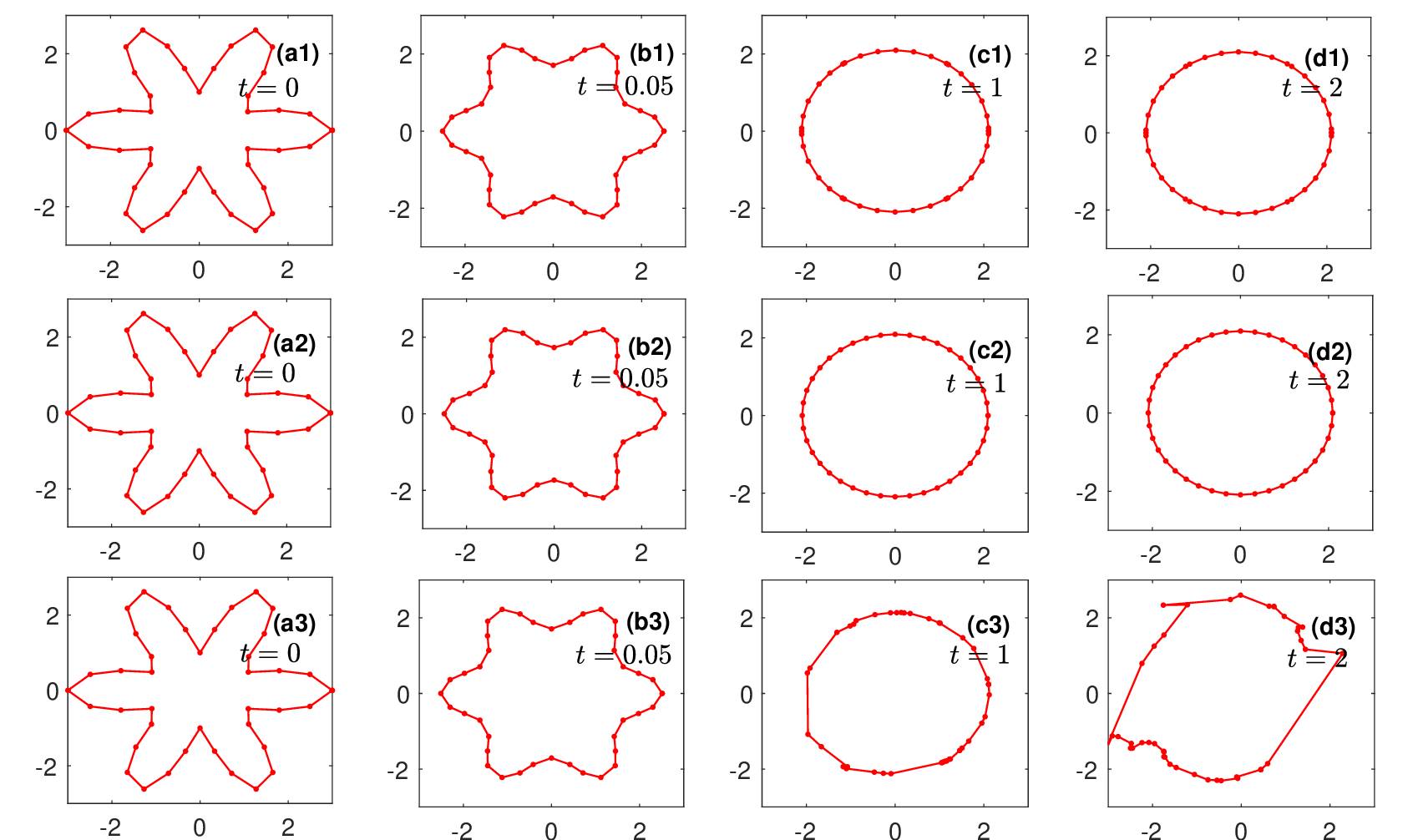}
\vspace{-1mm}
\caption{Evolution of the curve driven by SDF starting with Shape 4 as initial data by using the BGN2 scheme (shown in the first row), the equi-BGN2 scheme (shown in the second row) and without mesh regularization procedure (shown in the third row).} The simulations are conducted with a grid number of $N=40$ and a time step size $\tau=1/180$.
\label{Fig:flower_evo_MRandnoMR}
\end{figure}

Finally, we close this section by simulating the evolution of a nonconvex initial curve~\cite{Mikula-Sevcovic2004,Balazovjech-Mikula,Mackenzie-Nolan-Rowlatt-Insall}
driven by CSF, AP-CSF and SDF using the BGN2 schemes. The initial curve can be parametrized as
\[
\bX(\rho)=(
	 \cos(2\pi \rho),
	\sin(\cos(2\pi\rho ))+\sin(2\pi \rho)(0.7+\sin(2\pi \rho)\sin^2(6\pi \rho))),
\]
for $\rho\in \mathbb{I}=[0,1]$. The numerical results are depicted in Fig.~\ref{Fig:Mikula_evo}.
As shown in this figure, the CSF initially transforms the intricate curve into a circle before it disappear. Both the AP-CSF and SDF
drive the curve to evolve into a perfect circle as its equilibrium shape.

\begin{figure}[htp!]
\hspace{-7mm}
\includegraphics[width=5.3in,height=3in]{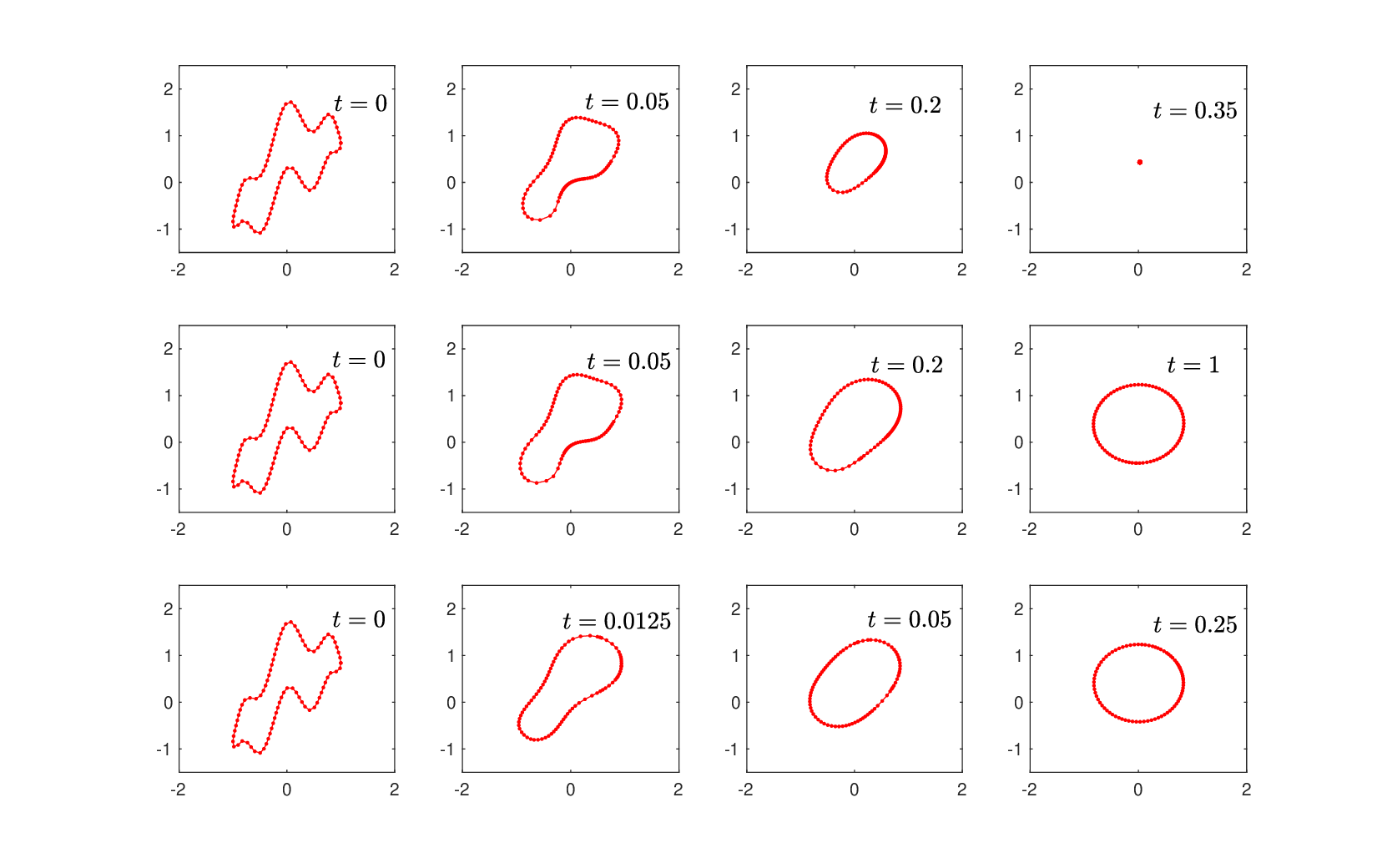}
\vspace{-13mm}
\caption{Snapshots of the curve evolution using the proposed BGN2 schemes for three distinct geometric flows: CSF (first row), AP-CSF (second row) and SDF (third row).  The simulations are conducted with $N=80$ and $\tau=1/640$.}
\label{Fig:Mikula_evo}
\end{figure}

%

\section{Conclusions}

We proposed two novel temporal second-order, BGN-based parametric finite element methods  (i.e., the BGN2 and the equi-BGN2 schemes) for solving different geometric flows of curves (e.g., CSF, AP-CSF and SDF). Based on the BGN formulation and the corresponding semi-discrete FEM approximation, our numerical methods employ a Crank-Nicolson leap-frog method to discretize in time. The key idea lies in choosing a discrete inner product over the curve $\Gamma^m$, such that the time level $t_m$ coincides with the time at which all quantities have approximations with an error of $\mathcal{O}(\tau^2)$. We established the well-posedness
of the BGN2 scheme under some suitable assumptions. Additionally, we showed that the equi-BGN2 scheme is unconditional energy-stable. We emphasized the use of shape metrics (manifold distance and Hausdorff distance) rather than function norms (e.g., $L^2$-norm, $L^{\infty}$-norm) to measure numerical errors of BGN-based schemes.
 In the case of certain initial curves, such as a `flower' shape, we found that the BGN2 scheme (resp. the equi-BGN2 scheme), in conjunction with the BGN1 scheme (resp. the equi-BGN1 scheme) for mesh regularization, exhibited remarkable  stability in practical simulations.
Extensive numerical experiments demonstrated that the proposed BGN2 and equi-BGN2 schemes achieve second-order accuracy in time, as measured by the shape metrics, outperforming the BGN1 scheme in terms of accuracy.

Furthermore, it is worth mentioning that the approach we have presented for constructing a temporal high-order BGN-based scheme can be readily extended to address various other problems, such as anisotropic geometric flows~\cite{Bao-Jiang-Li}, Willmore flow~\cite{BGN08C}, two-phase flow~\cite{Garcke23}, solid-state dewetting~\cite{Zhao-Jiang-Bao2021} and geometric flows in 3D~\cite{Zhao-Jiang-Bao}.

 In our future research, we will further investigate the development of structure-preserving temporal high-order BGN-based schemes~\cite{Bao-Zhao,JiangLi21} and conduct the numerical analysis of the BGN-based schemes with respect to the shape metric. These investigations will contribute to enhancing the overall understanding and applicability of the BGN type scheme in different contexts.

\section*{CRediT authorship contribution statement}
{\textbf{Wei Jiang}}: Conceptualization, Methodology, Supervision, Writing. {\textbf{Chunmei Su}}: Conceptualization, Methodology, Supervision, Writing. {\textbf{Ganghui Zhang}}: Methodology, Numerical experiments, Visualization and Writing.

\section*{Declaration of competing interest}
The authors declare that they have no known competing financial interests or personal relationships that could have appeared to influence the work reported in this paper.

\section*{Data availability}
No data was used for the research described in the article.

\section*{Acknowledgement}
We would like to thank the anonymous reviewers for their helpful comments and suggestions. This work was partially supported by the NSFC 12271414 and 11871384 (W. J.), the Natural Science Foundation of Hubei Province Grant No. 2022CFB245 (W. J.), and NSFC 12201342 (C. S. and G. Z.). The numerical calculations in this paper have been done on the supercomputing system in the Supercomputing Center of Wuhan University.


\end{document}